\documentclass[12pt,a4paper]{amsart}
\usepackage{amsmath,amssymb,amsthm,amscd,mathrsfs}
\usepackage[text={160mm,230mm},centering,headsep=2mm,footskip=5mm]{geometry}
\usepackage[dvipsnames]{xcolor}
\usepackage[colorlinks=true,linkcolor=RoyalBlue,citecolor=Mahogany,pagebackref]{hyperref}
\usepackage{graphicx}
\usepackage{tikz-cd}
\usepackage{tikz}
\usepackage{tikzsymbols}
\usepackage{mathtools}
\usetikzlibrary{patterns}
\usepackage[all]{xy}
\SelectTips{cm}{10}
\everyxy={<2.5em,0em>:} 
\usepackage{stmaryrd}

\newcounter{CountAlpha}

\theoremstyle{theorem}

\newtheorem{Thm}{Theorem}[section]
\newtheorem{Lem}[Thm]{Lemma}
\newtheorem{Cor}[Thm]{Corollary}
\newtheorem{Prop}[Thm]{Proposition}

\theoremstyle{definition}
\newtheorem{Def}[Thm]{Definition}
\newtheorem{Ex}[Thm]{Example}
\newtheorem{Rk}[Thm]{Remark}

\newtheorem{Obs}[Thm]{Observation}

\newtheorem{Constr}[Thm]{Construction}
\newtheorem{Con}[Thm]{Construction}

\newcommand{\car}{\operatorname{char}}

\newcommand{\Spec}{\operatorname{Spec}}
\newcommand{\Proj}{\operatorname{Proj}}

\newcommand{\ini}{\operatorname{in}}
\newcommand{\inv}{\operatorname{inv}}
\newcommand{\Int}{{\operatorname{Int}}}

\newcommand{\Ker}{\operatorname{ker}}

\newcommand{\K}{k}

\newcommand{\gqz}{{\geq 0}}
\newcommand{\Bl}{{\operatorname{Bl}}}

\newcommand{\IA}{\mathbb{A}}

\newcommand{\IF}{\mathbb{F}}
\newcommand{\IG}{\mathbb{G}}
\newcommand{\IP}{\mathbb{P}}
\newcommand{\IR}{\mathbb{R}}
\newcommand{\IQ}{\mathbb{Q}}
\newcommand{\IZ}{\mathbb{Z}}

\newcommand{\Gm}{\mathbb{G}_m}
\newcommand{\Gmu}{\pmb{\mu}}

\newcommand{\cI}{\mathcal{I}}

\newcommand{\cO}{\mathcal{O}}

\newcommand{\cX}{\mathcal{X}}

\newcommand{\be}{\mathbf{e}} 
\newcommand{\bw}{\mathbf{w}} 
\newcommand{\bv}{\mathbf{v}} 
\newcommand{\bu}{\mathbf{u}} 

\newcommand{\q}{\; / \;}
\newcommand{\qq}[1]{\; /_{#1} \;}

\newcommand{\mm}{\mathfrak{m}}

\definecolor{gruen}{rgb}{0, 0.55, 0}

\numberwithin{equation}{section}

\newcommand{\double}{\genfrac..{0pt}1
{\raise -1pt\hbox{$\scriptstyle\longrightarrow$}}{\raise 3pt\hbox
{$\scriptstyle\longrightarrow$}}}

\begin{document}

\title{Resolving plane curves using stack-theoretic blow-ups}



\author{Dan Abramovich}
\address{\tiny Department of Mathematics, Brown University, Box 1917,
	Providence, RI~02912, USA}
\email{dan\_abramovich@brown.edu}

\author{Ming Hao Quek}
\address{\tiny Harvard University, Department of Mathematics, 1 Oxford Street, Cambridge, MA 02138}
\email{mhquek@math.harvard.edu}

\author{Bernd Schober}
\address{\tiny  None
(Hamburg, Germany)}
\email{schober.math@gmail.com}

\date{\today}

\begin{abstract}
	
	Stack-theoretic blow-ups have proven to be efficient in resolving singularities over fields of characteristic zero.
	In this article, we move forward towards positive characteristic where new challenges arise. 
	In particular, the dimension of the tangent space of the Artin stack created after a weighted blow-up may increase, which makes it hard to apply  inductive arguments --- even if the maximal order decreases.
	We focus on the case of curve singularities embedded into a smooth surface defined over a perfect field.
	For this special situation we propose a solution to overcome the inductive challenge through canonically constructed multi-weighted blow-ups.

\end{abstract}

\maketitle

\setcounter{tocdepth}{1}
\tableofcontents

\section{Introduction}

This paper is concerned with \emph{resolution of singularities}, a challenging subject in algebraic geometry, especially when considered in positive and mixed characteristics. We focus however on the case of plane curves,  known since the early days of the subject, beginning with Newton. Many methods are described in \cite[Chapter 1]{Kollar}. A recent result on curves, with a similar spirit to the present paper, is \cite{FGM}.

A new method of resolution in charactertistic 0, using stack-theoretic weighted blow-ups, was introduced in \cite{ATW_weighted, McQuillan}, and one must wonder if this method has implications in positive characteristic, and what possible hurdles it encounters. 
Some of these questions are answered in \cite{Abramovich-Schober}, for binomial ideals, and in our work in preparation \cite{AQS}, where we construct an analogue of the weighted invariant of \cite{ATW_weighted} and describe its properties.

The purpose of this article is to identify this invariant for plane curves over a perfect field, apply the associated weighted blow-up of the plane, and see what happens. First we have:

\begin{Thm}
	\label{Thm:Main1} 
	Let $ C \subset S $ be a curve embedded in a regular surface $ S $ over a perfect field $k$.  
	Let $ q \in C $ be a singular closed point
	such that the reduction $ C_{\rm red} $ of $ C $ is singular at $ q $. 
	\begin{enumerate} 
	\item\label{It:well-defined} There is a well-defined, unique center $J=(x_1^{a_1}, x_2^{a_2})$ of maximal invariant $(a_1,a_2)$ admissible for $ C \subset S $ at $q$,
	where $ a_1 $ is the order of $ C$ at $ q $. 
	\item\label{It:birational} Let $\bar J = (x_1^{1/w_1}, x_2^{1/w_2})$ be the associated reduced center. Its stack-theoretic weighted blow-up $S' = Bl_{\bar J}(S)\to S$ is a proper birational morphism from a smooth Artin stack $S'$ which is a global quotient of a variety by $\IG_m$, in particular tame.
	\item\label{It:order-drops} The order of the proper transform at every point of the blow-up lying over $ q $ is strictly smaller than $ a_1 $.
	\end{enumerate}
\end{Thm}

The  invariant entry $ a_2 $, which we use to define the center,  is derived from Hironaka's characteristic polyhedron (Section~\ref{Sec:Hiro}).

For further discussion in the introduction, we mention that we have $ a_1 \leq a_2 $ and hence $ w_1 \geq w_2 $ 
and $ w_1 , w_2 $ are coprime integers.

We note that Statement \ref{It:order-drops} is not true when using standard blow-ups, an advantage in any characteristic.
It also suggests that any reasonable singularity invariant drops.\footnote{Strictly speaking,  we can only speak of a full invariant, and prove this, once \cite{AQS} is completed.} 
This looks promising for an inductive argument for resolution, and works when the characteristic does not intervene, see Theorem \ref{Thm:Main2}\eqref{It:invariant-drops-nop}.

However, when the characteristic does intervene, one cannot proceed inductively precisely because, in general, the resulting two-dimensional Artin stack modifying $S$ has tangent space of dimension 3. This would require a treatment in arbitrary embedding dimension, which at present we do not have, even in the cases at hand.

We propose a remedy here which replaces the Artin stack weighted blow-up by a Deligne--Mumford multi-weighted blow-up, and which has all desired properties.

\begin{Thm}
	\label{Thm:Main2} 	 
	With the assumptions of Theorem \ref{Thm:Main1}, let $p$ be the characteristic of $k$.
	\begin{enumerate} 
	\item\label{It:invariant-drops-nop} If $p\,{\nmid}\, w_1w_2$ the blow-up is a Deligne--Mumford stack
	with everywhere smaller invariants: $\inv_C'(q')<(a_1,a_2)$ for every point $q'$ on the proper transform $C'$ lying over $q$.
	\item\label{It:invariant-drops-p} If $p\,|\,w_1w_2$ there is a canonical multi-weighted blow-up $S'' \to S$, 
	a proper birational map from a smooth Deligne--Mumford stack with
	strictly smaller order at every point lying above $ q $ 
		and
	everywhere smaller invariant as above. 
\end{enumerate}
Iterating $S'' \to S$, one obtains in finitely many steps a resolution of singularities by tame Deligne--Mumford stacks. An application of Bergh's destackification \cite{Bergh} provides a further stack-theoretic modification  $S^{(3)} \to S''$ whose coarse moduli space $\bar S^{(3)}$ is smooth, inducing a resolution $\bar S^{(3)} \to S$ by a variety.
\end{Thm}

\medskip

Let us briefly outline the idea for the construction of the multi-weighted blow-up in Theorem~\ref{Thm:Main2}(\ref{It:invariant-drops-p}) in the setting $ S = \IA_\K^2 $.
Here, the (multi-)weighted blow-up corresponds to a fantastack as explained in \cite[Example~2.10]{AQ}, see 
Construction~\ref{Con:multi-weighted-blow-up}.
The construction of a fantastack arises in our setting from a full fan in $ \IR_{\geq 0}^2 $. 
The weighted blow-up with reduced center $\bar J = (x_1^{1/w_1}, x_2^{1/w_2})$ corresponds to the fan with rays $ \be_1 = (1,0), \be_2 = (0,1) $ and $ \bw = (w_1,w_2) $.
The latter is illustrated in the following picture:
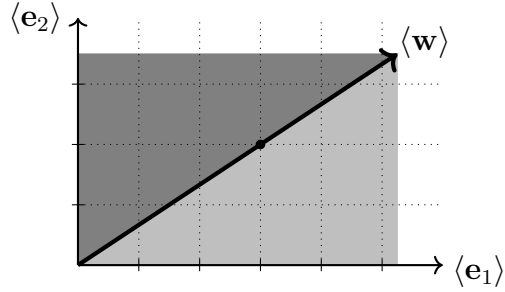
\[
\begin{tikzpicture}[scale=0.8]
	
	\filldraw[gray] (0,0)--(5.25,3.5)--(0,3.5);
	\filldraw[lightgray] (0,0)--(5.25,3.5)--(5.25,0)--(0,0);
	
	\draw[->, thick] (0,0)--(0,4.1); 
	\draw[->, thick] (0,0)--(6,0); 
	
	\node at (6.6,-0.1) {$\langle \be_1 \rangle$};
	\node at (-0.7,4.1) {$\langle \be_2 \rangle$};
	\node at (5.7,3.7) {$\langle \bw \rangle$};
	
	\foreach \x in {1,...,5} 
	\draw (\x,0.1) -- (\x,-0.1);
	
	\foreach \y in {1,...,3} 
	\draw (0.1,\y) -- (-0.1,\y);
	
	
	\foreach \x in {1,...,5}
	\draw[dotted] (\x,0) -- (\x,4.1);
	\foreach \y in {1,...,3}
	\draw[dotted] (0,\y) -- (6,\y);
	
	\draw[ultra thick, ->] (0,0)--(5.25,3.5);
			
		\filldraw (3,2) circle (2pt);

\end{tikzpicture} 
\] 
The blow-up is covered by two charts which are determined by two neighboring rays. 
The challenge of having an Artin stack which is not a Deligne--Mumford stack arises if the determinant of the $ 2 \times 2 $ matrix given by the minimal generators of the rays is zero in $ \K $.
In other words, if either 
\[
	\det \begin{pmatrix}
		1 &  w_1 \\ 0 & w_2 
	\end{pmatrix}
 = w_2 = 0 \mod p 
 \hspace{1cm}
 \mbox { or } 
 \hspace{1cm}
 \det \begin{pmatrix}
 	w_1 & 0  \\ w_2 & 1  
 \end{pmatrix}
 = w_1 = 0 \mod p .
\]
This explains the condition $p\,{\nmid}\, w_1w_2$, resp.~$ p\,|\,w_1w_2 $, in Theorem~\ref{Thm:Main2}.

The idea for the construction of the multi-weighted blow-up in Theorem~\ref{Thm:Main2}\eqref{It:invariant-drops-p} in the given special case is to refine the fan by introducing an additional ray, so that none of the relevant determinants is zero in $k$:

If $ w_1 \equiv 0 \mod p $, we introduce the ray determined by the vector
\(
		\bu := (1,1). 
\)
We get  the picture:
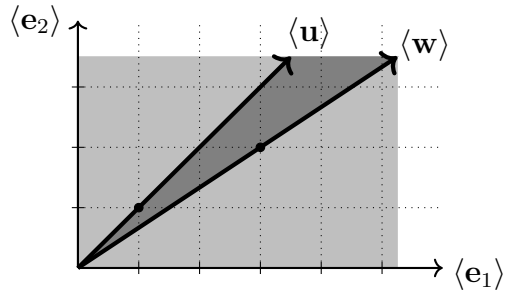
\[
	\begin{tikzpicture}[scale=0.8]
		
		\filldraw[lightgray] (10,0)--(13.5,3.5)--(10,3.5)--(10,0);
		\filldraw[gray] (10,0)--(13.5,3.5)--(15.25,3.5)--(10,0);
		\filldraw[lightgray] (10,0)--(15.25,3.5)--(15.25,0)--(10,0);
		
		\draw[->, thick] (10,0)--(10,4.1); 
		\draw[->, thick] (10,0)--(16,0); 
		
		\node at (16.6,-0.1) {$\langle \be_1 \rangle$};
		\node at (9.3,4.1) {$\langle \be_2 \rangle$};
		\node at (15.7,3.7) {$\langle \bw \rangle$};
		\node at (13.8,3.9) {$\langle \bu \rangle$};
		
		\foreach \x in {11,...,15} 
		\draw (\x,0.1) -- (\x,-0.1);
		
		\foreach \y in {1,...,3} 
		\draw (10.1,\y) -- (9.9,\y);
		
		
		\foreach \x in {11,...,15}
		\draw[dotted] (\x,0) -- (\x,4.1);
		\foreach \y in {1,...,3}
		\draw[dotted] (10,\y) -- (16,\y);
		
		\draw[ultra thick, ->] (10,0)--(13.5,3.5);
		\draw[ultra thick, ->] (10,0)--(15.25,3.5);
				
		\filldraw (13,2) circle (2pt);
		\filldraw (11,1) circle (2pt);

	\end{tikzpicture} 
\]
We observe that the respective determinants are
\[
	\det \begin{pmatrix}
	1 &  w_1 \\ 0 & w_2 
\end{pmatrix} = w_2,
	\ \
	\det \begin{pmatrix}
		w_1 & 1  \\ w_2 & 1  
	\end{pmatrix}
	=  w_1 - w_2 , 
	\ \ 
		\det \begin{pmatrix}
		1 &  0 \\ 1 & 1 
	\end{pmatrix}
	= 1
\]
and all of them are non-zero modulo $ p $ since since $ p\nmid\ w_2$  and $ p\, |\, w_1$.

On the other hand, if $ w_2 \equiv 0 \mod p $,
we apply an analogous but more tricky construction,
by taking into account that the roles of $ w_1 $ and $ w_2 $ are switched.
We introduce
\[
\widetilde\bu := (\kappa,1), \qquad \kappa := \lceil w_1/w_2 \rceil \geq 1.
\]

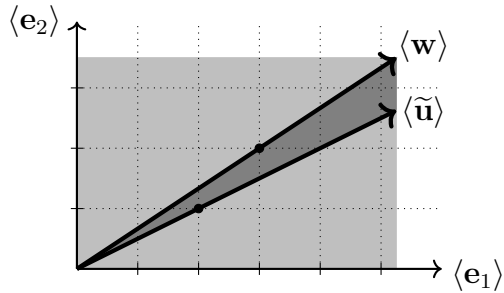
\[
	\begin{tikzpicture}[scale=0.8]
		
		\filldraw[lightgray] (10,0)--(15.25,3.5)--(10,3.5)--(10,0);
		\filldraw[gray] (10,0)--(15.25,2.625)--(15.25,3.5)--(10,0);
		\filldraw[lightgray] (10,0)--(15.25,2.625)--(15.25,0)--(10,0);
		
		\draw[->, thick] (10,0)--(10,4.1); 
		\draw[->, thick] (10,0)--(16,0); 
		
		\node at (16.6,-0.1) {$\langle \be_1 \rangle$};
		\node at (9.3,4.1) {$\langle \be_2 \rangle$};
		\node at (15.7,3.7) {$\langle \bw \rangle$};
		\node at (15.7,2.6) {$\langle \widetilde\bu \rangle$};
		
		\foreach \x in {11,...,15} 
		\draw (\x,0.1) -- (\x,-0.1);
		
		\foreach \y in {1,...,3} 
		\draw (10.1,\y) -- (9.9,\y);
		
		
		\foreach \x in {11,...,15}
		\draw[dotted] (\x,0) -- (\x,4.1);
		\foreach \y in {1,...,3}
		\draw[dotted] (10,\y) -- (16,\y);
		
		\draw[ultra thick, ->] (10,0)--(15.25,3.5);
		\draw[ultra thick, ->] (10,0)--(15.25,2.625);
		
		\filldraw (13,2) circle (2pt);
		\filldraw (12,1) circle (2pt);
	
	\end{tikzpicture} 
\]

We note that in the previous case $\lceil w_2/w_1 \rceil = 1$, so the two cases are consistent.
In this case we observe that the respective determinants are
\[
	\det \begin{pmatrix}
	w_1 & 0 \\  w_2 & 1 
\end{pmatrix} = w_1,
	\ \
	\det \begin{pmatrix}
		\kappa &w_1   \\ 1 &w_2   
	\end{pmatrix}
	=  \kappa w_2 - w_1 , 
	\ \ 
		\det \begin{pmatrix}
		1 &  \kappa \\ 0 & 1 
	\end{pmatrix}
	= 1
\]
and all of them are non-zero modulo $ p $ since $ p\nmid\ w_1$   and $ p\, |\, w_2$.

\begin{Rk}[{Non-embedded resolution}] If one is only interested in resolving $C$ and not keeping it embedded, some of the computations become simpler. As Example \ref{Ex:one} suggests, in the first chart the proper transform of the curve does not pass through the point on the blow-up of $S$ where the stabilizer is non-trivial. However this avoids only the easy part of our construction --- for the other chart we need the full force of multi-weighted blow-ups.
\end{Rk}

\subsection*{Outline of the article}
In Section \ref{Sec:Hiro} we review Hironaka's characteristic polyhedron in the case of plane curves. This leads to the choice of parameters and the definition of invairants in Section \ref{Sec:inv}. The weighted center is introduced in Section \ref{Sec:center}, its properties proven in Theorem \ref{Thm:J_unique}, implying Theorem \ref{Thm:Main1}\eqref{It:well-defined}. The weighted blow-up is introduced in Section \ref{Sec:weighted}, and shown to be birational and tame, proving Theorem \ref{Thm:Main1}\eqref{It:birational}. The transform of $C$ under the weighted blow-up is discussed   in Section \ref{Sec:transform}, and we show that the order drops, proving Theorem \ref{Thm:Main1}\eqref{It:order-drops}, as well as  Theorem \ref{Thm:Main2}\eqref{It:invariant-drops-nop}. The multi-weighted blow-up is introduced in Section \ref{Sec:multi}, in particular proving that it is a well-defined Deligne--Mumford stack. Theorem \ref{Thm:Main2}\eqref{It:invariant-drops-p} is completed in Section \ref{Sec:drops}.

\subsection*{Acknowledgements}

This research work was supported in part by funds from BSF grant 2022230, NSF grants DMS-2100548, DMS-2401358, and Simons Foundation Fellowship SFI-MPS-SFM-00006274.
Dan Abramovich thanks IHES  for its hospitality during Fall 2024.

We thank Andr\'e Belotto da Silva, Ang\'elica Benito, Raymond van Bommel, Ana Bravo, Vincent Cossart, Santiago Encinas, Mike Montoro, James Myer,  Michael Temkin and Jaros{\l}aw W{\l}odarczyk for discussions on the topic.

\section{Hironaka's characteristic polyhedron}
\label{Sec:Hiro}

In this section, we recall the notion of Hironaka's characteristic polyhedron \cite{HiroCharPoly} in our particular setting of plane curves. 
For the reader's convenience, we try to avoid technicalities as much as possible. 
Aside from \cite{HiroCharPoly}, other references discussing Hironaka's characteristic polyhedron in full generality are \cite[Chapter~8 and its appendix]{CJS} or \cite[Section~1]{CS-Compl}, for example.

\medskip

Let $ C \subset S $ be a curve embedded in a regular surface $ S $.  
Let $ q \in C $ be a singular closed point
such that the reduction $ C_{\rm red} $ of $ C $ is singular at $ q $. 
We consider the local situation at $ q $.
Let $ \cO := \cO_{S,q} $ be the regular local ring of $ S $ at $ q $
with maximal ideal $ \mm \subset \cO $ and residue field $ \K \simeq \cO/\mm $.
Let $ (x,y) $ be a regular system of parameters for $ \cO $. 
We assume that $ \cO $ contains the residue field $ \K $ and that $ \K $ is perfect.
The situation reduces to a hypersurface singularity $ \{ f = 0 \} $ given by an element $ f \in \mm^2 $.
In the $ \mm $-adic completion $ \widehat{\cO} \simeq \K[[x,y]] $ of $ \cO $, $ f $ has an expansion%
\footnote{In fact, there exists a {\em finite} expansion in $ \cO $ , 
	$ f = \sum_{(a,b)} \epsilon_{a,b} x^a y^b $ with $ \epsilon_{a,b} \in \cO^\times \cup \{0\} $, by Noetherianity of $ \cO $, see \cite[Proposition~2.1]{Cossart-Pitant} for example.
All notions that we define in the following can be introduced using this expansion and thus do not depend on passing to the completion unless otherwise specified.
For the reader's convenience, we stick to an expansion with coefficients in $ \K $.}
\[
	f = \sum_{(a,b) \in \IZ_\gqz^2} c_{a,b} x^a y^b,  
	\ \  \mbox{ with } c_{a,b} \in \K .
\]
In this case, the multiplicity of $ C $ at $ q $ (resp., of $ f $ at $ \mm $) is given by 
\[
	\nu := \nu_q (C) := \nu_\mm(f) := \inf \{ a+ b \mid c_{a,b} \neq 0  \} .
\]
Since $ C $ is singular at $ q $, we have $ \nu \geq 2 $.

The {\em initial form of $ f $ at $ \mm $} is defined as
\[
	\ini_\mm (f) := \sum_{(a,b) \,:\, a+b = \nu} c_{a,b} X^a Y^b \in \operatorname{gr}_\mm(\cO) := \bigoplus_{s \geq 0} \mm^s/ \mm^{s+1} \simeq \K [ X, Y],
\]
where where $ X $ (resp., $ Y $) denotes the image of $ x $ (resp., $ y $) in $ \mm/\mm^2 $.
In the following, we assume that there exist $ \lambda, \mu \in \K $ such that  
\begin{equation}
	\label{eq:dir_dim_one}
	\ini_\mm (f) = ( \lambda X + \mu Y )^\nu . 
\end{equation}
The case in which \eqref{eq:dir_dim_one} is not true will be discussed in Remark~\ref{Rk:dir_dim_zero}.
After a linear coordinate change and possibly renaming of the variables, 
we may suppose without loss of generality 
that $ \lambda = 0 $ and $ \mu = 1 $. 
In other words, the assumption \eqref{eq:dir_dim_one} becomes
\[
	\ini_\mm(f) = Y^\nu .
\]

\begin{Def}\label{Def:projected polyhedron}
	In the given setup, the {\em projected polyhedron $ \Delta(f;x;y) $ of $ (f;x;y) $} is defined as 
	the smallest closed convex set containing all points of the set
	\[
		\left\{ \frac{a}{\nu - b} + \IR_\gqz  \ \Big| \  c_{a,b} \neq 0 \right\}. 
	\]
	In other words,  $ \Delta(f;x;y) $ is the semi-line $ [ v ; \infty ) $ starting at the vertex $ v := v(f;x;y) := \inf\{ \frac{a}{\nu - b} \mid c_{a,b} \neq 0  \} $ and going to infinity. 
\end{Def}

The name projected polyhedron is coming from the fact that it is a projection of the Newton polyhedron, which is the two-dimensional polyhedron determined as the closed convex hull of the set $ \{ (a,b) + \IR_\gqz^2 \mid c_{a,b} \neq 0 \} $, from the point $ (0, \nu) $ to the $ x $-axis, followed by a homothety with factor $ 1/\nu $.

\begin{Obs}
	\label{Obs:Proj_poly_non-empty}
	Notice that $ \Delta(f;x;y) $ is non-empty:
	In the empty case, we must have $ f = \epsilon y^\nu $ for some unit $ \epsilon \in \cO $,
	which is a contradiction to
	the assumption that $ C_{\rm red} $ is singular at $ q $.  
	
	Moreover, 
	by the assumption $ \ini_\mm(f) = Y^\nu $, 
	we have that the vertex of $ \Delta(f;x;y) $ is a positive rational number strictly bigger than $ 1 $, i.e., the vertex {$ v > 1 $.} 
\end{Obs}

The vertex $ v \in \Delta (f;x;y) $ is called {\em solvable} if there exists some $ \lambda \in \K \smallsetminus \{ 0 \} $ such that,
if we set $ y' := y - \lambda x^v $, then $ v \notin \Delta(f;x;y') $. 
Notice that this implies $  \Delta(f;x;y) \supsetneq  \Delta(f;x;y') $.

\begin{Ex}
	Let $ \K= \mathbb{F}_3 $ and $ f = -y^{9} + y^3 - x^6 + x^{11} $.
	We see that $ \Delta(f;x;y ) $ is the semi-line with vertex $ v = 2 $.
	The vertex $ v $ is solvable:
	if we introduce $ y' := y - x^2 $,
	then $ f = -y'^9 + y'^3 - x^{18} + x^{11} $. 
	In particular, $ \Delta(f;x;y' ) $ is the semi-line with vertex $ v'= 11/3 $. 
	Observe that $ v' $ is not solvable. 
\end{Ex}

\begin{Def}
	\label{Def:char_poly}
	Let $ f \in \cO $
	with $ \ini_\mm(f) = Y^\nu $ for $ \nu := \nu_\mm(f) $.
	The {\em characteristic polyhedron $ \Delta(f;x) $ of $ (f;x) $} is defined as the intersection over all projected polyhedra $ \Delta(f;x;y') $ with $ y' $ varying through all choices such that $ (x,y') $ is a regular system of parameters (r.s.p.) for $ \cO $,
	\[
		\Delta(f;x) = \bigcap_{y' \,:\, (x,y')\ \rm r.s.p.} \Delta(f;x;y').  
	\]
\end{Def}

Starting with $ \Delta(f;x;y) $, 
we may verify whether its vertex $ v $ is solvable. 
If it is so, we modify $ y $ appropriately and check whether the new vertex (which is strictly bigger than $ v $) is also solvable. 
By \cite[Theorem~(4.8)]{HiroCharPoly}, repeated solving of vertices leads to an element $ \widehat{y} \in \widehat{\cO} $ in the $ \mm $-adic completion $ \widehat{\cO} $ of $ \cO $ such that 
$
	\Delta(f;x;\widehat{y}) = \Delta(f;x).
$
Since $ \dim ( \Spec (R/ ( f ) ) ) = 1 $, 
\cite[Theorem~B]{CS-Compl} implies that it is not required to pass to the completion:

\begin{Prop}
	\label{Prop:char_poly_in_R}
	Let the setup be as in Definition~\ref{Def:char_poly}.
	There exists an element $ z \in \cO $ such that $ (x,z) $ is a regular system of parameters for $ \cO $ and  
	\[
	\Delta(f;x;z) = \Delta(f;x).
	\]
\end{Prop}

In particular, the characteristic polyhedron $ \Delta(f;x) $ is non-empty, by Observation~\ref{Obs:Proj_poly_non-empty}. 
We denote by $ \delta := \delta(f;x) $ the vertex of the characteristic polyhedron, 
i.e.,
\[
	 \Delta(f;x) = [ \delta; \infty ) . 
\]

From its construction, it seems as if $ \delta $ depends on the embedding $ (f) \subset \cO $, 
but in fact, \cite[Corollary~B(3)]{CJSch} (see also \cite[Theorem~4.18]{CS-dim2})
provides:

\begin{Prop}
	\label{Prop:delta_invariant}
	The number $ \delta \in \IQ_{>1} $ is an invariant of the singularity $ \Spec (\cO/ (f) ) $.
\end{Prop}

Associated to the vertex $ \delta $, there is also the notion of the {\em $ \delta $-initial form $ \ini_\delta (f) $ of $ f $}, which is determined by the initial form of $ f $ at $ \mm $ and the terms of the expansion contributing to the vertex $ \delta $,
	\[
	\ini_\delta (f) := Y^\nu + \sum_{(a,b) \, : \, \frac{a}{\nu- b} = \delta} c_{a,b} X^a Y^b.
	\]
In \cite[Corollary~B(3)]{CJSch}, it is proven that the graded ring $ \operatorname{gr}_\delta (\cO/(f)) \simeq \K[X,Y]/(\ini_\delta (f)) $ is an invariant of the singularity, thus it is not depending on an embedding.

\begin{Rk}
	\label{Rk:dir_dim_zero}
	At the beginning of the section, we imposed the assumption \eqref{eq:dir_dim_one}. 
	If this is not true, 
	then one can verify that 
	the projected polyhedron $ \Delta(f;x;y) $ is the semi-line starting at $ v = 1 $
	and no change in $ y $ can eliminate the vertex.
	Hence, we set $ \Delta(f;x) := \Delta(f;x;y) = [1; \infty )  $ and $ \delta = \delta(f;x) := 1 $
	and we have $ \ini_\delta (f) = \ini_\mm (f) $.
\end{Rk}

\begin{Def}
	\label{Def:Fnu_Fdelta}
	Let $ f \in \cO $ be as above. 
	Since $ \K[x,y] \subset \cO $, we introduce the following lifts to $ \cO $ of the initial form of $ f $ at $ \mm $ and of the $ \delta $-initial form of $ f $:
	\[
		\begin{array}{l}
			\displaystyle 
			F_\nu := \sum_{(a,b) \, :\, a+b = \nu} c_{a,b} x^a y^b \in \cO ,
			\\[20pt]
			\displaystyle 
			F_\delta := F_\nu  + \sum_{(a,b) \, : \, \frac{a}{\nu- b} = \delta} c_{a,b} x^a y^b  \in \cO . 
		\end{array} 
	\]
\end{Def}

\begin{Ex}
	\label{Ex:picture}
	Consider the hypersurface $ \{ f = 0 \}  $ given by $ f = z^7 - x^6 z^2 + x^{15} \in \IF_p [x,z] $ (for any prime $p $) and $ q $ the origin.
	We have $ \Delta(f;x;z) = \Delta(f;x)  $,
	$ \delta = \frac{6}{5} $
	and 
	\[  
		F_\delta = z^7 - x^6 z^2 = z^2 (z^5 - x^6) .
	\]
	Observe that $ F_\delta $ is not reduced
	and the reduction of the curve $ \{ F_\delta = 0 \} $ is singular at the origin.
	In the following picture, we show the Newton polyhedron of $ f $ 
	and mark its projection which leads to the characteristic polyhedron after a homothety, see Figure \ref{Fig:projected}.
\begin{figure}[h] 
	\[
	\begin{tikzpicture}[scale=0.5]
		
		\filldraw[lightgray] (0,9.2)--(0,7)--(6,2)--(15,0)--(16.2,0) -- (16.2,9.2);
		
		\draw[->] (0,0)--(0,9.5); 
		\node at (-0.4,9.7) {$  z $};
		
		\draw[->] (0,0)--(16.5,0); 
		\node at (17, 0) {$ x $};
		
		\foreach \x in {1,...,16} 
		\draw (\x,0.1) -- (\x,-0.1);
		
		\foreach \y in {1,...,9} 
		\draw (0.1,\y) -- (-0.1,\y);
		
		
		\foreach \x in {1,...,16}
		\draw[dotted] (\x,0) -- (\x,9.5);
		\foreach \y in {1,...,9}
		\draw[dotted] (0,\y) -- (16,\y);
		
		\draw[ thick] (0,9.2)--(0,7)--(6,2)--(15,0)--(16.2,0);
		
		\draw[dashed, very thick] (6,2)--(8.4,0);
		
		\draw[[-, ultra thick] (8.4,0) -- (16.2,0); 
		
		\node at (10, -1) {$ \nu \delta =42/5$};
		\node at (-1.5, 7) {$ \nu =7 $};
		
	\end{tikzpicture}
	\] \caption{} \label{Fig:projected}
	\end{figure}
\end{Ex}

\section{Invariant and parameters}\label{Sec:inv}

We continue as above, with a curve singularity $ C $ on a smooth surface $ S $ and a singular point $q$. 
In the local ring $ \cO := \cO_{S,q} $ with maximal ideal $ \mm $ and perfect residue field $ \K $, the curve $ C $ is given as $ \{ f = 0 \} $. 
We assume $f$ is not a pure power of a regular parameter,
in other words, the reduction of $ C $ is singular at $ q $, as before.
The previous section provided an integer $\nu$ --- the order of $C$ at $q$ --- and a rational number $\delta$ --- indicating that the \emph{characteristic polyhedron} is the interval $[\delta,\infty)$, which are invariants of the singularity.

From here on we combine these two numbers in the \emph{invariant} 
$$ \inv_C(q) := \inv_f(q) := (a_1,a_2) = (\nu, \delta\nu).$$ It indicates that the \emph{Newton polyhedron} of $f$ with respect to the parameters $(x_1,x_2) := (z,x)$ indicated in the previous section has the affine span of  $\left( (a_1,0), (0,a_2) \right)$ as supporting hyperplane along an edge, cf.~Example~\ref{Ex:picture}.
Note that we switched the order of the parameters. The reason for this is that we aim to be consistent with the existing literature.

The value set of the invariant is ordered lexicographically and as such it is well-ordered: while $a_2$ is rational, its denominator lies in the finite set  $\{1,\ldots, a_1-1\}$.

To recall, the parameters $(x,z)$ were chosen as follows: 
\begin{itemize}
	\item Expanding $ f = F_\nu + h $ at $q$ with $ F_\nu $ introduced in Definition~\ref{Def:Fnu_Fdelta} and $ h \in \mm^{\nu +1} $, 
	if $ F_\nu $ is \emph{not} a pure $ \nu$-power of a parameter, any pair of regular parameters $(x,z)$ works and $\delta = 1$.
	\item Otherwise, 
	\begin{itemize} 
		\item $z^\nu\equiv f \mod \mm^{\nu+1} $, in particular the residue  $z \mod  \mm^2  $ is uniquely defined.
		\item $x$ is an arbitrary parameter so that $(x,z)$ is a regular system, namely we have an equality of ideals $ \mm = (x,z)$,
		\item $z$ is chosen so that $\Delta(f;x;z) = \Delta(f;x)$, but otherwise arbitrarily.
	\end{itemize}	
	In other words, $z$ maximizes $\delta(f;x;z)$.
	\item The value $\delta$ is characterized by the fact that the vertex of the characteristic polyhedron is not solvable and it is an invariant of the singularity of $ C $ at $ q $ (Proposition~\ref{Prop:delta_invariant}). 
\end{itemize}
We note that the parameters $(x_1,x_2) = (z,x) $ are not uniquely defined. A natural question that arises is: 
\begin{quote} 
	Is there a canonical structure, independent of choices, determined by the situation?
\end{quote}

We answer this question next.

\section{The center}\label{Sec:center}

Given any parameters $(x_1,x_2)$ of the regular local ring $ ( \cO, \mm, \K ) $ and positive rational numbers $ a_1, a_2 \in \IQ_+ $ we use the symbol $J = (x_1^{a_1},x_2^{a_2})$ to denote the monomial valuation $v_J\colon \cO \smallsetminus \{0\} \to \IQ$ uniquely determined by $v_J(x_1) = \frac1{a_1}, v_J(x_2) = \frac1{a_2}.$
Explicitly, given $ g \in \cO \smallsetminus \{ 0 \}$ with an expansion $g(x_1,x_2) = \sum d_{i_1,i_2} x_1^{i_1}x_2^{i_2} \neq 0$ we have 
$$v_J(g) = \min\left\{\frac{i_1}{a_1}+ \frac{i_2}{a_2}\ \big| \  d_{i_1,i_2} \neq 0\right\}.$$

The notation $J = (x_1^{a_1},x_2^{a_2})$ is chosen to indicate that $x_1^{a_1}$ and $x_2^{a_2}$ both have valuation 1. Once again, the data of $(x_1,x_2)$ and $(a_1,a_2)$ determines the valuation but is not uniquely determined by it. 
\begin{Rk}\label{Rk:local} In \cite[Section~2.2]{ATW_weighted}, a more global notion of \emph{valuative $\IQ$-ideals} is used to describe centers. Since we work locally, valuations suffice.
\end{Rk}

For the sake of ordering, we restrict attention to centers with $a_1\leq a_2$. This does not limit the generality since we may switch $a_1$ with $a_2$ if necessary.

The notation is also suggestive in the following way: as in \cite[Section~5.2]{ATW_weighted}, we say that a center $J$ is \emph{admissible} for $f$ if $v_J(f) \geq 1$. This is equivalent to saying that $f\in (x_1^{a_1},x_2^{a_2})^{\Int}$.

\begin{Thm}
	\label{Thm:J_unique}
	 Let $\{f=0\}$ be a curve singularity embedded in a surface over a perfect field, as before, such that $ f $ is not a pure power of a parameter. Set 
	$$(a_1,a_2) := (\nu,\nu\delta), \ x_1 := z, x_2 :=\ x , $$ 
	where $ \Delta(f;x;z) = \Delta(f;x) = [\delta; \infty) $, as before. Then the center $J = (x_1^{a_1},x_2^{a_2})$ is uniquely determined. It is the unique center admissible for $f$ for which $(a_1,a_2)$ attains its maximum for the lexicographical order.
\end{Thm}

We note that this theorem implies Theorem \ref{Thm:Main1}\eqref{It:well-defined}.
\begin{proof}
	First, observe that
	\begin{equation}
		\label{eq:lexmax} 
		\max_{\geq_{\rm lex}} \{ (b_1, b_2) \in \IQ_+^2 \mid \exists \, (y_1,y_2) :  (y_1^{b_1},y_2^{b_2}) \mbox{ admissible center for } f  \}
		\geq (a_1, a_2)
	\end{equation} 
	since $ J = (x_1^{a_1},x_2^{a_2}) = (z^\nu,x^{\nu\delta}) $ is admissible by the definition of the characteristic polyhedron.
	 Indeed, given an expansion\footnote{Note that $d_{i_1,i_2}=c_{b,a}$ with  switched indices from the previous section.} 
	 $$f(z,x) =    \sum d_{i_1,i_2} z^{i_1}x^{i_2}, $$ the projection $i_2/(a_1-i_1) $ of any $(i_1,i_2)$ with $d_{i_1,i_2} \neq 0$ must lie in $[a_2/a_1,\infty)$ ---  this follows from Definition \ref{Def:projected polyhedron}, taking into account that $d_{i_1,i_2}$ are switched from  that context. The inequality   $i_2/(a_1-i_1) \geq a_2/a_1$ is equivalent to  $i_1/a_1 + i_2/a_2 \geq 1$.

	{\bf Uniqueness of $ \boldsymbol{a_1} $:} 
	Let $(y_1^{b_1},y_2^{b_2})$ be an admissible center with $b_1 \leq b_2$.
		Since the order of $ f $ at $ (y_1,y_2) $ is $ \nu $, 
		there exists a monomial $ y_1^a y_2^{\nu-a} $ with non-zero coefficient in an expansion of $ f $, for some $ a \in \{0, \ldots, \nu \} $.
		Since $(y_1^{b_1},y_2^{b_2})$ is admissible and $b_1 \leq b_2$, 
	we have $b_1 \leq \nu=a_1$. 
	Hence, this shows 
		\[
			\max \{ b_1 \in \IQ_+ \mid \exists \, b_2 \in \IQ_+ \, \mbox{ and } \, (y_1,y_2) :  (y_1^{b_1},y_2^{b_2}) \mbox{ admissible center for } f  \}
		= a_1.
		\]
	
	{\bf Uniqueness of $ \boldsymbol{a_2} $:} 
	Next, let us have a look at $ b_2 $. 
		We already know that any admissible center with maximal $ b_1 $ is of the form $ (y_1^{\nu},y_2^{b_2})$.
	By the definition of the characteristic polyhedron, we have $\delta(f;y_2;y_1) \leq \delta$.
	Hence, in an expansion $ f = f(y_1, y_2) $, there exists a monomial $ y_1^{j_1}y_2^{j_2} $ with non-zero coefficient and $ j_2 / (\nu - j_1) \leq \delta $. 
	(Clearly, we have $ j_2 \neq 0 $ and $ j_1 < \nu $ here.)
Since $ (y_1^{\nu},y_2^{b_2})$ is admissible for $ f $, we deduce from this that $ b_2 \leq \nu \delta = a_2 $ 
proving that \eqref{eq:lexmax} is an equality:
	Due to the admissibility, we have 
	$ {j_1}/{\nu} + {j_2}/{b_2}  \geq 1 $
	for the chosen $ (j_1, j_2) $. 
	This can be rewritten as 
	$
		b_2 \leq \nu {j_2}/{(\nu - j_1)}
	$
	which implies $ b_2 \leq \nu \delta $, as desired.

	{\bf Uniqueness of the center:} 
	It remains to prove that the chosen center is independent of the choice of the parameters $ (x_1, x_2) $.
	Let $J_x = (x_1^{a_1},x_2^{a_2})$ and $J_y=(y_1^{a_1},y_2^{a_2})$ be admissible centers for $f$. We claim that $J_x = J_y$. 
	We may assume $a_1<a_2$, otherwise the equality of centers follows directly, as in that case $v_{J_x} = v_{J_y} = a_1 \nu_\mm $.
	
	\emph{Choosing the same $x_2$.} 
	First note that, given $x_1$, the parameter $x_2$ can be chosen to be any parameter such that $(x_1,x_2) =  \mm$, and similarly for $y_2$. There is always a parameter $x_2$ so that both $(x_1,x_2) = (y_1,x_2) = \mm  $ since the projective space $\IP(\mm/\mm^2)$ has at least 3 points. We may thus assume $x_2 = y_2$. In particular $v_{J_y}(x_2) = v_{J_x}(x_2) =  1/a_2$. 
	
	\emph{The valuations of $x_1$ agree.}  
	We claim that $v_{J_y}(x_1) = 1/a_1$. 
	If $ x_1 $ and $ y_1 $ only differ by the multiplication by a unit,
		the statement is clear. 
		Hence, assume this is not the case. 
		Since $ (x_1, x_2) $ and $ (y_1, x_2) $ are systems of regular parameters for $ \cO $, 
		an expansion of $ x_1 $ in terms of $ (y_1, x_2 ) $ is of the form 
	$ x_1 = \epsilon y_1 + \mu x_2^A $, for some $ A \in \IZ_+ $ and units $ \epsilon, \mu $.
	As $ v_{J_y} $ is a monomial valuation, we get 
	\[ 
		v_{J_y} (x_1) = \min \left\{ \frac{1}{a_1}, \frac{A}{a_2}  \right\} \leq \frac{1}{a_1} .
	\]
	Arguing by contradiction assume the inequality is strict. 
	Considering the expansion of $f$ in terms of $(x_1,x_2)$, 
	all nonzero terms $d_{i_1,i_2} x_1^{i_1}x_2^{i_2}$ satisfy $i_2/a_2 \geq (a_1 - i_1) / a_1 $
	since $ J_x $ is admissible for $ f $, and equality holds
	for $i_1 = a_1, i_2=0$. 
	The value with respect to the valuation $ v_{J_y} $ is
	 $v_{J_y}(d_{i_1,i_2} x_1^{i_1}x_2^{i_2}) = i_1v_{J_y}(x_1) + i_2/a_2 \geq a_1 v_{J_y}(x_1)$,
	 with equality if and only if $i_1 = a_1, i_2 = 0$. Thus  the term $d_{a_1,0}x_1^{a_1}$, which is nonzero by the assumption $a_1 < a_2$,  has strictly smaller value than all other terms. 
	 By the non-Archimedean property we have $v_{J_y}(f)= a_1 v_{J_y}(x_1) < 1$, giving a contradiction. Hence $v_{J_y}(x_1) = 1/a_1$, and similarly $v_{J_x}(y_1) = 1/a_1$.
	
	\emph{The valuations coincide.} For any integer $k$, the integrally closed ideals $(x_1^{ka_1},x_2^{ka_2})^\Int$ and $(y_1^{ka_1},y_2^{ka_2})^\Int$ of functions of value $\geq k$ with respect to $J_x,J_y$, are contained in each other, hence are equal. 
	In conclusion, the valuations coincide, $J_x = J_y$. 
\end{proof}

Write  
\begin{equation}
	\label{eq:expansion} 
	f =f_{J,1} + \widetilde h,
\end{equation} where all monomials of $f_{J,1}$ have valuation 1 and $ \widetilde h \in \cO $ with $ v_J (\widetilde h) > 1 $.
Notice that $ f_{J,1} = F_\delta $ (Definition~\ref{Def:Fnu_Fdelta}), but in order to emphasize the connection to the valuation $ v_J $, we use in the following the symbol $ f_{J,1} $.
 
\begin{Cor}\label{Cor:leading-term} The polynomial $f_{J,1}$ is not a $\nu$-th power of a regular parameter. We have $\inv_{f_{J,1}}(q) = (a_1,a_2)$ with the same center $(x_1^{a_1},x_2^{a_2})$. \end{Cor}
\begin{proof}  Indeed the value $\delta$ has the property that $f_{J,1}(x_1,x_2)$ is not a pure power of a parameter,
	since $ \delta \in \Delta (f; x_2 ;x_1) $ is not solvable.
	 Hence the vertex $\delta$ of the projected polyhedron $\Delta(f_{J,1}; x_2;x_1)  = [\delta; \infty )  $ is not solvable because $ \ini_\delta (f_{J,1})  = \ini_\delta (f) $.
 \end{proof}

\section{Weighted blow-ups}\label{Sec:weighted}

In this section, we recall the notion of weighted blow-ups and discuss particularities arising in positive characteristic.

\begin{Def}[{\cite[Section~2.4]{ATW_weighted}}]
		\label{Def:reduced_center}
		Given a center $J = (x_1^{a_1},x_2^{a_2})$,
		the associated \emph{reduced center} is defined as 
		$\bar J = (x_1^{1/w_1}, x_2^{1/w_2})$, 
		where $w_1,w_2 \in \IZ_+ $ are relatively prime integers 
		such that 
		\[ 
		\ell \cdot \left(\frac1{w_1},\frac1{w_2}\right) = (a_1,a_2)
		\] 
		for some integer $\ell \in \IZ_+ $.   
		In particular $\ell = a_1w_1 = a_2w_2$ is the least integral common integer multiple of $a_1,a_2$.
\end{Def}

We note that $v_{\bar J} = \ell v_J$.

\medskip

Recall that the usual blow-up of a regular center $ C := V (I) \subset S  $ is
	determined by the Rees algebra $ \bigoplus_{n\geq 0} I^n $;
	more precisely, the blow-up is the proper birational morphism 
	$ \operatorname{Bl}_C (S) := \Proj ( \bigoplus_{n\geq 0} I^n ) \to S $.
	In the weighted setting, we use the following generalizations
	which already appear in Rees's original work \cite{Rees}.

\begin{Def}
	Let $J = (x_1^{a_1},x_2^{a_2})$ with reduced center $\bar J = (x_1^{1/w_1}, x_2^{1/w_2})$. 
	For $ n \in \IZ $, set
	\[
		\cI_n := \{g \in \cO \mid  v_{\bar J}(g) \geq n\}. 
	\]

	\begin{enumerate}
		\item 
		The \emph{Rees algebra associated to $\bar J$} is
		$ R := \bigoplus_{n\geq 0} \cI_n$.  
		
		\item 
		The \emph{extended Rees algebra associated to $\bar J$} is 
		$\tilde R = \bigoplus_{n\in \IZ} \cI_n$. 
	\end{enumerate}
\end{Def}
We note that $\cI_n = \cO$ for $n\leq 0$. 
Moreover, both $R$ and $\tilde R$ are naturally $\cO$-algebras, justifying the terms.

In view of Remark \ref{Rk:local}, this construction globalizes, with $\cI_n = \cO_S$ away from $q$.

The algebra $\tilde R$ has a nice local presentation
\cite[Proposition~5.2.2]{Quek-Rydh}:
\begin{equation}
	\label{eq:presentation}
	\tilde R =  \cO[s,x_1',x_2'] / (x_1 - s^{w_1} x_1',x_2 -s^{w_2} x_2'),
\end{equation}
where $ s \in \cI_{-1} $ is the element $ 1 \in \cO $, and 
$ x_i' \in \cI_{w_i} $ corresponds to the element $ x_i $ in degree $ w_i $, for $ i \in \{1, 2\} $.

The variable $s$ is the exceptional variable, whose vanishing locus will define the exceptional divisor once we define the blow-up. 
The variable $x_i'$ is called the transformed variable corresponding to $x_i$, for $ i \in \{ 1, 2 \} $.

\begin{Constr}[{Degenerating to the weighted normal cone}]
	\label{Constr:Degen}
	 We write $B = \Spec_S(\tilde R)$ for the spectrum relative to the scheme $ S  $ of the quasi-coherent sheaf of $ \cO_S $-algebras $ R $. 
It has a structure morphism  $B \to S$ but also a morphism  $B \to \IA_\K^1 = \Spec k[s]$. This second morphism makes $B$ into the degeneration of $S$ to the weighted normal cone of $\tilde J$: 
\begin{itemize}
	\item Its fiber over $s=1$ is naturally isomorphic to $S$ as $x_1=x_1', x_2=x_2'$. 
	\item Its fiber over $s=0$ lies over the point $(x_1,x_2) = (0,0)$; it is isomorphic to $\Spec k[x_1',x_2']$ with the natural action of $\IG_m$ with weights $w_i$ on $x_i'$, making it the weighted normal cone of the point in $S$.
\end{itemize}

The grading provides an action of $\IG_m = \Spec k[t,t^{-1}]$ on $B$, explicitly 
\begin{equation}
	\label{eq:action}
	t \cdot (s,x_1',x_2') \quad = \quad (t^{-1} s, t^{w_1} x_1', t^{w_2} x_2').
\end{equation} 
Note that the action stabilizes the locus $V(\oplus_{n>0} \cI_n) = V(x_1', x_2')$.
\end{Constr}

\begin{Def}[Weighted blow-up]\label{Def:weighted}
	Using the notation of Construction~\ref{Constr:Degen},
	we define $ B_+ := B \smallsetminus V(x_1', x_2') $, 
	which is equipped with its natural $\IG_m$-action \eqref{eq:action}.
	\\
	The {\em weighted blow-up of $S$ at $\bar J$} is defined as the Artin stack
	\[ 
		\pi \colon \operatorname{Bl}_{\bar J} (S) := [B_+/\IG_m] \to S,
	\]
	where the morphism $ \pi $ is induced by the structure morphism $ B \to S $.
\end{Def}

Since the locus $ V(x_1', x_2') $ is removed, 
the weighted blow-up of $ S $ at $ \bar J $ is covered by two charts,
where in the first one $ x_1' $ is invertible (we also say $ x_1' \neq 0 $ in this case) 
while in the second $ x_2' $ is invertible (i.e., $ x_2' \neq 0 $).
For an explicit example for this, we refer to Example~\ref{Ex:one} below.

\medskip

The notation $B$ comes from W{\l}odarczyk's interpretation of $B$ as a birational cobordism connecting $S$ and its weighted blow-up \cite{Jarek_cobordism}.

\medskip

The following properties hold, in particular showing Theorem \ref{Thm:Main1}\eqref{It:birational}.

\begin{Lem}
	Let $ J = (x_1^{a_1},x_2^{a_2}) $
	and let $ \pi \colon \operatorname{Bl}_{\bar J} (S) = [B_+/\IG_m] \to S $	be the weighted blow-up of $S$ at the reduced center $\bar J = (x_1^{1/w_1}, x_2^{1/w_2})$.
	We have:
	\begin{enumerate}
		\item 
		All stabilizers of the Artin stack  $ \operatorname{Bl}_{\bar J} (S) $ are naturally subgroups of $\IG_m$.
		In particular, $ \operatorname{Bl}_{\bar J} (S) $ is a {\em tame} Artin stack.
		
		\item 
		The dimension of the quotient $ \operatorname{Bl}_{\bar J} (S) $ is two
		and $ \pi $ is an isomorphism outside the exceptional divisor, which is defined on the smooth covering of $ B_+ $ by $ s = 0 $.
	\end{enumerate}	
	In other words, the weighted blow-up $ \pi $ is a birational morphism.  
	
	\begin{enumerate}
		\item[(3)]
		At $s = x_1'=0$ the stabilizer group is $\Gmu_{w_2}$; 
		while at at $s= x_2'=0$ the stabilizer group is $\Gmu_{w_1}$.
		Here, $ \Gmu_m $ denotes the group of $ m $-th roots of unit for $ m \in \IZ_+ $.
	\end{enumerate}
\end{Lem}

\begin{proof} For (1) see \cite[Lemma~1.1.2]{Quek-Rydh}. For (2) 
		see \cite[Remark~3.2.10]{Quek-Rydh}. Finally (3) 
		 follows from \cite[Lemma~1.1.2]{Quek-Rydh}.
\end{proof}

\begin{Rk} 
	Ending up with an Artin stack may seem like a big hurdle --- after all, when stabilizers are divisible by the characteristic, charts for an Artin stack involve an inseparable cover.  
	
	However the theory of destackifcation, as developed by Bergh \cite{Bergh} and Bergh--Rydh \cite{Bergh-Rydh}
	shows that, if $\cX$ is a tame separated Artin stack, with possibly singular moduli space $X$, there is always a \emph{destackification}, 
	which in its weakest form means a resolution of singularities $X'\to X$ of the coarse moduli space.
	
	More precisely, destackification provides a sequence of operations of smooth stacks  $\cX' = \cX_n \to \cdots \to \cX_1 \to \cX_0 = \cX$, each either a smooth blow-up or a root construction along a smooth divisor, such that both $\cX'$ \emph{and its coarse moduli space} $X'$ are smooth. This means that  the canonical morphism $X' \to X$ induced by functoriality of coarse moduli spaces is indeed a resolution. 
		The process is functorial, in particular not changing the locus where the stack structure of $\cX$ is trivial.
	
	In our situation, where stabilizers are subgroups of a torus, all actions are linearizable, and therefore all singularities are toric, 
	and the algorithm is combinatorial and explicit (if computationally expensive). 
	
	As an example, the quotient of $\IA_\K^2$ by the diagonal action of $ \Gmu_p $ defined by 
	$(x,y) \mapsto (\zeta x, \zeta^{-1}y)$
	(with $ p := \car(\K) > 0 $) 
	has coarse moduli space with equation $ uv=t^p $, clearly a toric singularity which is eminently resolvable.
	
	In general, destackification does not mix well with algorithms of resolution of singularities, in the sense that it does not respect natural invariants of singularities. It is therefore usually carried out after stack-theoretic resolution is completed. One can view our multi-weighted blow-up introduced in Section \ref{Sec:multi} below as an exception. It is a very partial destackification, taking out just the inseparable stabilizers, while invariants continue to drop. 
\end{Rk}

\begin{Rk}[{\'Etale slices are good, flat slices are not}]
	Let $ p = \car(\K) \geq 0 $ be the characteristic of the residue field of $ S $ at $ q $.
	
	If $p\nmid w_2$ (which includes the case $ p = 0 $)
	then the chart $x_2' \neq 0$ has an \'etale slice given by $x_2' = 1$, which gives an \'etale presentation of the chart in $ \operatorname{Bl}_{\bar J} (S)$ as a Deligne--Mumford stack; a similar situation happens on  $x_1'\neq 0$ if $p\nmid w_1$. 
	
	If $p \mid w_i$ the corresponding chart only has a flat slice, which does not preserve singularities --- this is seen in Example~\ref{Ex:one} 
	below. 
	This means that analyzing such chart forces us to work with a smooth presentation as a $\IG_m$-quotient, with three-dimensional tangent space.
\end{Rk}

\begin{Ex}
	\label{Ex:one}
	Let $ \K = \IF_p $, for $ p $ prime, and $ S = \IA_\K^2 $.
	We consider the weighted blow-up of $ S $ at 
	\[
		\bar J = (x_1^{1/(p+1)}, x_2^{1/p}) .
	\] 
	This comes up, for instance, when resolving the curve $C = \{x_1^p = x_2^{p+1}\}$.
	The presentation \eqref{eq:presentation} provides that the corresponding extended Rees algebra is 
	\[ 
		\tilde R = \cO[s,x_1',x_2']\  /\  (x_1 - s^{p+1} x_1',x_2 -s^{p} x_2').
	\]
	
	In the chart where $ x_1' $ is invertible, we may take the \'etale slice $x_1'=1$ of the action, providing the equations 
	$x_1 = s^{p+1}, x_2 = s^px_2'$ with the action of $\Gmu_{p+1}$ given by
	$(s,x_2') \mapsto (\zeta_{p+1}^{-1}\cdot s, \zeta_{p+1}^{-1}\cdot x_2')$, 
	as $ \zeta_{p+1}^p = \zeta_{p+1}^{-1}$, where $ \zeta_{p+1} $ is a primitive $(p+1)$-st root of unity. 
	This gives an isomorphism of this chart
	$$
		\left[\Spec_S\left(\cO_S[s,x_1',x_2', {x_1'}^{-1}]\ /\ (x_1 - s^{p+1} x_1',x_2 -s^{p} x_2')\right)\   \bigg/\  \IG_m \right]
		\ \simeq \
		\left[\Spec k[s,x_2']/\Gmu_{p+1}\right].
	$$ 
		Anticipating the next section, the proper transform
of $C$ is given by 
	$ \{{x_2'}^{p+1} = 1 \}$, a smooth curve in this chart, which does not even intersect the fixed-point  locus 
	$ \{ s=x_2'=0\} $ 
	of the action. Therefore, this chart is as well-behaved as one would wish: it provides a schematic resolution of the curve $C$ right away.
	
	On the other hand, notice that in the chart 
	$$\left[\Spec_S\left(\cO_S[s,x_1',x_2', {x_2'}^{-1}]\ /\  (x_1 - s^{p+1} x_1',x_2 -s^{p} x_2')\right)\ \ \bigg/\ \ \IG_m \right]$$ 
	where $ x_2' $ is invertible, 
	we have  an isomorphism of the  chart with  $[\Spec k[s,x_1']/\Gmu_{p}]$, the quotient of a smooth scheme by $\Gmu_p$, the action given by $(s,x_1') \mapsto (\zeta_p^{-1}s, \zeta_p x_1')$. It is a slice in the \emph{flat} topology of the $\IG_m$-action. The proper transform $C'$ of the curve $C$ is given by 
	$ \{ {x_1'}^p = 1 \}$. 
	This singularity has invariant $(p,\infty)$, strictly bigger than the invariant $(p,p+1)$ of the curve $C$ we started with. What is going on?
	
	The answer is that, since $\Gmu_p$ is not a smooth scheme, this chart does not know much about the singularity! In this case the action of $\Gmu_p$ on the curve 
	$ \{ {x_1'}^p = 1 \}$ is free, and the quotient is indeed smooth, but the chart does not tell this story.
	
	Instead, the action of $\IG_m$ on the scheme $B_+ = \Spec_S(\widetilde R) \smallsetminus V(x_1',x_2')$ provides a chart in the \emph{smooth} topology, which does know about singularities. The equation of $C'$ on this chart is 
	$\{{x_1'}^p = {x_2'}^{p+1}\}$. 
	It looks exactly like the original equation, since it is weighted-homogeneous. Note however that we removed the locus 
	$ \{ x_1'=x_2'=0 \}$, 
	so the curve $C'$ is indeed smooth. 
\end{Ex}

\section{Transforming $f$ by the weighted blow-up} \label{Sec:transform}

After introducing the invariant and the corresponding reduced center, we show next that the weighted blow-up makes the order of a curve singularity drop strictly.
Due to the increase of the dimension of the tangent space,
this is not sufficient for a desingularization and we are facing an induction challenge, see Observation~\ref{Obs:challenge}.
This motivates our turn towards multi-weighted blow-ups in the remainder of the article.

Recall that we consider a curve $ C $ embedded in a surface $ S $ and singular at a closed point $ q \in C $ at which the residue field is perfect.
Locally at $ q $, the curve is given as $ \{ f = 0 \} $, for some element $ f \in \cO = \cO_{S,q} $.
Expand again $ f = f_{J,1}(x_1,x_2) + \widetilde h(x_1,x_2) $ as in Equation \eqref{eq:expansion}. 
Pulling back to $B$ we have 
\[
	f(x_1,x_2) 
	= f(s^{w_1} x_1',  s^{w_2} x_2') 
	= s^\ell f_{J,1}(x_1',x_2') + s^{\ell+1} \widetilde h' (s, x_1',x_2') 
\]
with obvious notation $ \widetilde h' (s, x_1',x_2') $ and where $ \ell\in\IZ_+ $ is the positive integer introduced in Definition~\ref{Def:reduced_center}.
The fact that the term $ s^\ell $ (resp.~$ s^{\ell+1} $) factors from the pullback of $ f_{J,1} $ (resp.~$ \widetilde h) $ follows from the definition of these terms in Equation \eqref{eq:expansion}.

\begin{Def}\label{Def:proper}
	The \emph{proper transform of $ f $ 
	on $B$} is defined as 
	\[ 
		f'(s, x_1',x_2') 
		:= s^{-\ell} f(s^{w_1} x_1',  s^{w_2} x_2')  
		= f_{J,1}(x_1',x_2') + s \, \widetilde h (s, x_1',x_2').
	\]
\end{Def}

Restricting the transform to the fiber $s=1$, which is isomorphic to $S$, one obtains $f(x_1,x_2)$, which has invariant $(a_1,a_2)$ at the origin and smaller invariant  nearby. The same holds for all fibers where $s \neq 0$.
Restricting to the fiber $s=0$, which is the weighted normal cone, we obtain $f_{J,1}(x_1',x_2') $. 
By Corollary \ref{Cor:leading-term} we have again invariant  $(a_1,a_2)$ at the origin and smaller invariant nearby.  In other words, \emph{the invariant on any fiber of $B_+ \to \IA_\K^1$ drops.}

We can show more:

\begin{proof}[The order drops: Proof of Theorem \ref{Thm:Main1}\eqref{It:order-drops}]
Since the order $\nu$ is functorial for smooth morphisms, showing that the order drops on $B_+$ implies that it drops on the quotient $\Bl_{\bar J}(S) = [B_+/\IG_m]$. 
Note that $\nu_q(f) \leq \nu^{\log}_q(f)$, where the logarithmic order $\nu^{\log}$ is defined with respect to the exceptional divisor $\{s=0\}$: it is the order where $s\neq 0$ and otherwise it is the order of the restriction to $\{s=0\}$, namely the order of $f_{J,1}(x_1',x_2')$.

When $s\neq 0$ we simply are computing the order of $f$ away from $\{x_1=x_2=0\}$, which is $<\nu=a_1$ since $f$ is not a pure $\nu$-th power.

When $s=0$ we are computing the order of $f_{J,1}(x_1',x_2')$ away from $\{x_1'=x_2'=0\}$. But once again, $f_{J,1}(x_1',x_2') = x_1^{\nu} + \cdots$ is not a pure $\nu$-th power, see Corollary \ref{Cor:leading-term}. Thus its order of vanishing at any point on this locus is $<\nu=a_1$, as needed.
\end{proof}

\begin{Rk} In the case where $a_1=a_2$ and $w_1=w_2 = 1$ the fact that the order drops in the standard blow-up is well known. It is a consequence of a result of Hironaka--Mizutani, see \cite[Theorem~3.14]{CJS}
and \cite[Claim~4.2]{CS-Compl}.
\end{Rk}

	\begin{Rk}
		In the forthcoming work \cite{AQS} 
		we develop the invariant of ideals in higher dimension and show its smooth functoriality and upper semi-continuity on $k$-points. The first term in the invariant is always the order. 
		In particular this shows for plane curves that the invariant of $f'$ is $(a_1,a_2)$ along $ \{ x_1' = x_2' = 0 \}$ 
		and lower at points nearby. 
		Thus such invariant on $B_+$ is strictly lower than $(a_1,a_2)$, hence the
		invariant of $f'$ on the blow-up is strictly lower than the invariant of $f$.
	\end{Rk}

\begin{Obs}[The induction challenge]
	\label{Obs:challenge}
	The decrease of the invariant on $B_+$ and of the order on $\Bl_{\bar J}(S)$ would imply resolution of singularities for curves embedded in a surface if we could proceed by induction. 
	There is a problem though: 
	where the blow-up is not a Deligne--Mumford stack, the weighted blow-up is presented as a quotient of a 3-dimensional scheme, so induction does not apply, even though there is a $\IG_m$-action reducing to a 2-dimensional Artin stack. This is seen, for example, in
	\cite[Example~5.1]{Wlodarczyk-Rees} due to W{\l}odarczyk.

	We remark that, if $p\nmid w_i$ (which is in particular true if $ \car(\K) = 0 $) 
		one has a corresponding Deligne--Mumford chart by a surface, so the present argument applies directly without need of generalizing.
\end{Obs}

Our goal here is to show that there is a way to resurrect the induction for curve singularities that can be embedded in a surface using a certain \emph{multi-weighted blow-up}
without having to introduce the invariant in higher dimension.

\section{Multi-weighted blow-ups}\label{Sec:multi} 
Since the inductive challenge only appears in positive characteristics, we assume from here on that the characteristic $p$ is positive.
Consider a tuple $\bw = (w_1,w_2)$
of coprime positive integers, and assume $p$ divides $w_1$
(so $p$ does not divide $w_2$). 
In contrast to earlier assumptions, we no longer assume that $w_1>w_2$. This allows us to treat the cases $p\,|\,w_i$ as one.

As noted earlier,  
the stack-theoretic weighted blow-up of $S$ 
along the reduced center
$\bar J= (x_1^{1/w_1},x_2^{1/w_2}) $
is an Artin stack, and not a  Deligne--Mumford stack. In this section, we propose a slightly different construction which resolves this issue.

We start from the case where $S = \IA_\K^2$, the affine plane, where we can use toric methods, in particular Geraschenko and Satriano's fantastacks to give a concrete description.

Recall from \cite[Example 2.10]{AQ} that the stack-theoretic weighted blow-up of $\IA_K^2$ at $\bar J$ can be viewed as the fantastack 
$\cX_{(\Sigma,\beta)}$ (in the sense of \cite[Section 4]{Toric1}) associated to the data $(\Sigma,\beta)$, where: \begin{enumerate}
	\item $\Sigma$ is the fan on $N = \IZ^2$ generated by the full-dimensional cones $\langle \be_1,\bw \rangle$ and $\langle \be_2,\bw \rangle$ in $\IR^2_{\geq 0}$, and
	\item letting $\Sigma(1)$ denote the set of rays in $\Sigma$, $\beta$ is the homomorphism $\IZ^{\Sigma(1)} \to N$ mapping each $\be_\rho$ (for $\rho \in \Sigma(1)$) to the first lattice point on the ray $\rho$.
\end{enumerate}
Since $\beta$ is canonically associated with $\Sigma$ in this case, we shall write $\cX_{(\Sigma,\beta)} = \cX_\Sigma$.

\begin{Con}\label{Con:multi-weighted-blow-up}
	To fix the aforementioned issue, we consider a new fan $\Sigma'$ on $N$ generated by the cones $\langle \be_1,\bw \rangle$, $\langle \bw,\bu \rangle$ and $\langle \be_2,\bu \rangle$ in $\IR^2_{\geq 0}$, where 
	\[
	\bu := (1,\kappa), \qquad \kappa := \lceil w_2/w_1 \rceil
	.
	\]
	As before, let $\Sigma'(1)$ denote the set of rays in $\Sigma'$, and let $\beta' \colon \IZ^{\Sigma'(1)} \to N$ be the homomorphism canonically associated with $\Sigma'$, which sends each $\be_\rho$ (for $\rho \in \Sigma'(1)$) to the first lattice point on the ray $\rho$. 
	
	We consider the fantastack $\cX_{\Sigma'} := \cX_{(\Sigma',\beta')}$ associated to $(\Sigma',\beta')$. To describe this fantastack, let $\widehat{\Sigma}'$ be the fan on $\IZ^{\Sigma'(1)}$ generated by the two-dimensional cones $\langle \be_{\be_1},\be_{\bw} \rangle$, $\langle \be_\bw,\be_\bu \rangle$, and $\langle \be_{\be_2},\be_{\bu} \rangle$ in $\IR^4_{\geq 0}$. 
	Observe that $\beta$ is compatible with the fans $\widehat{\Sigma}'$ and $\Sigma'$, and $\Sigma'$ is a subdivision of the standard fan on $\IR_{\geq 0}^2$ that is generated by $\langle \be_1,\be_2 \rangle$. Therefore, there are toric morphisms of toric varieties $X_{\widehat{\Sigma}'} \to X_{\Sigma'} \to \IA_\K^2$. As noted in \cite[Section~3]{Toric1}, this descends to the following: 
	\[ 
	\pi \colon \cX_{\Sigma'} := \left[X_{\widehat{\Sigma}'} \q G_{\beta'} \right] \longrightarrow \IA_\K^2, \qquad \textrm{where} \quad G_{\beta'} := \Ker\left(\Gm^4 \xrightarrow{T_{\beta'}} \Gm^2\right).
	\]
	where in the above expression: \begin{enumerate}
		\item[(i)] $X_{\widehat{\Sigma}'}$ is the toric variety associated to the fan $\widehat{\Sigma}'$ on $\IZ^4$.
		\item[(ii)] $T_{\beta'}$ is the homomorphism of tori induced by $\beta'$.
		\item[(iii)] $G_{\beta'}$ acts on $X_{\widehat{\Sigma}'}$ as a subgroup of $\Gm^4$.
	\end{enumerate}
	We refer to $\pi$ as the \emph{multi-weighted blow-up} of $\IA_\K^2$ with respect to $\Sigma'$, cf. \cite[Definition 2.5]{AQ}. To explicate this multi-weighted blow-up, we follow \cite[\S 2.1.11]{AQ}. The homomorphism $\beta \colon \IZ^4 \to N = \IZ^2$ fits in the short exact sequence: \[
	0 \to \IZ^2 \xrightarrow{\alpha \ = \ \tiny \begin{bmatrix}w_1 & 1 \\ w_2 & \kappa \\ -1 & 0 \\ 0 & -1\end{bmatrix}} \IZ^4 \xrightarrow{\footnotesize \beta \ = \ \begin{bmatrix}1 & 0 & w_1 & 1 \\ 0 & 1 & w_2 & \kappa\end{bmatrix}} \IZ^2 \to 0
	\]
	from which one deduces the following: \begin{enumerate}
		\item We have the commutative diagram: \[
		\begin{tikzcd}
			X_{\widehat{\Sigma}'} = \IA_\K^4 \smallsetminus V(J_{\Sigma'}) \arrow[to=1-2, hookrightarrow, "\textrm{\tiny{open}}"] \arrow[to=2-1, twoheadrightarrow, swap, "\textrm{\tiny{stack-theoretic quotient}}"] & \IA_\K^4  \arrow[to=1-3] & \IA_\K^2 \\
			\cX_{\Sigma'} = \left[X_{\widehat{\Sigma}'} \q \Gm^2\right] \arrow[to=1-3, swap, bend right=15, "\pi"] & &
		\end{tikzcd}
		\]
		where we give $\IA_\K^4$ the coordinates \[
		\qquad \qquad \IA_\K^4 = \Spec\bigl(k[x_1',x_2',s,u]\bigr)
		\]
		so that the above morphism $\IA_\K^4 \to \IA_\K^2$ of affine spaces corresponds to the homomorphism $k[x_1,x_2] \to k[x_1',x_2',s,u]$ which maps \[
		x_1 \mapsto x_1's^{w_1}u \qquad \textrm{and} \qquad x_2 \mapsto x_2's^{w_2}u^\kappa.
		\]
		Moreover, $J_{\Sigma'}$ is the ideal $(x_2'u,x_1'x_2',x_1's) \subset k[x_1',x_2',s,u]$, called the irrelevant ideal.
		
		\item The action of $G_{\beta'} \simeq \Gm^2$ on $X_{\widehat{\Sigma}'} \subset \IA_\K^4 = \Spec(k[x_1',x_2',s,u])$ can be interpreted from the matrix $\alpha$ as follows: 
			$(t_1,t_2)\in \IG_m^2$ acts on $(x_1',x_2',s,u)$ via
			$$(x_1',x_2',s,u)\quad\mapsto\quad (t_1^{w_1}t_2\,x_1',\,t_1^{w_2} t_2^\kappa\,x_2',\,t_1^{-1}\,s,\,t_2^{-1}\,u).$$
		\item The Cartier divisors $E_s := \left[V(s) \q \Gm^2\right]$ and $E_u := \left[V(u) \q \Gm^2\right]$ in $\cX_{\Sigma'}$ are the \emph{exceptional divisors} of the multi-weighted blow-up $\pi$. The complement of $E_s \cup E_u$ in $\cX_{\Sigma'}$ maps isomorphically under $\pi$ onto $\IA_\K^2 \smallsetminus V(x_1,x_2)$.
	\end{enumerate}
	It is convenient to package the above observations in the following compact form: 
	\[
	\left[\Spec_{\IA_\K^2}\left(\frac{\cO_{\IA_\K^2}[x_1',x_2',s,u]}{(x_1's^{w_1}u - x_1,x_2's^{w_2}u^k - x_2)}\right) \smallsetminus V(x_2'u,x_1'x_2',x_1's) \qq\alpha \Gm^2\right] \xrightarrow{\pi} \IA_\K^2 = \Spec(k[x_1,x_2]),
	\]
	where the subscript $\alpha$ indicates the $\IG_m$-action above. For later use we introduce the notation $\tilde R'_0:=\cO_{\IA_\K^2}[x_1',x_2',s,u]/{(x_1's^{w_1}u - x_1,x_2's^{w_2}u^k - x_2)}$ for the algebra appearing above. Its global sections form the ring $k[x_1',x_2',s,u]$.
\end{Con}

\begin{Ex}
	\label{Ex:wb_vs_mulit_wb}
	If $p=3$ and $\bw = (3,2)$, we have $\bu = (1,1)$ and the  fans $\Sigma$, $\Sigma'$ of Figure~\ref{Fig:multiweighted}.
\begin{figure}[h]
	\[
	\begin{tikzpicture}[scale=0.8]
		
		\filldraw[gray] (0,0)--(5.25,3.5)--(0,3.5);
		\filldraw[lightgray] (0,0)--(5.25,3.5)--(5.25,0)--(0,0);
		
		\draw[->, thick] (0,0)--(0,4.1); 
		\draw[<->, thick] (0,0)--(6,0); 
		
		\node at (6.6,-0.1) {$\langle \be_1 \rangle$};
		\node at (-0.7,4.1) {$\langle \be_2 \rangle$};
		\node at (5.7,3.7) {$\langle \bw \rangle$};
		
		\foreach \x in {1,...,5} 
		\draw (\x,0.1) -- (\x,-0.1);
		
		\foreach \y in {1,...,3} 
		\draw (0.1,\y) -- (-0.1,\y);
		
		
		\foreach \x in {1,...,5}
		\draw[dotted] (\x,0) -- (\x,4.1);
		\foreach \y in {1,...,3}
		\draw[dotted] (0,\y) -- (6,\y);
		
		\draw[ultra thick, ->] (0,0)--(5.25,3.5);
		
		
		\node at (8,2) {\Large $\rightsquigarrow$};
		\node at (3,-0.7) {$\Sigma$};
		\node at (13,-0.7) {$\Sigma'$};
		
		\filldraw[lightgray] (10,0)--(13.5,3.5)--(10,3.5)--(10,0);
		\filldraw[gray] (10,0)--(13.5,3.5)--(15.25,3.5)--(10,0);
		\filldraw[lightgray] (10,0)--(15.25,3.5)--(15.25,0)--(10,0);
		
		\draw[->, thick] (10,0)--(10,4.1); 
		\draw[<->, thick] (10,0)--(16,0); 
		
		\node at (16.6,-0.1) {$\langle \be_1 \rangle$};
		\node at (9.3,4.1) {$\langle \be_2 \rangle$};
		\node at (15.7,3.7) {$\langle \bw \rangle$};
		\node at (13.8,3.9) {$\langle \bu \rangle$};
		
		\foreach \x in {11,...,15} 
		\draw (\x,0.1) -- (\x,-0.1);
		
		\foreach \y in {1,...,3} 
		\draw (10.1,\y) -- (9.9,\y);
		
		
		\foreach \x in {11,...,15}
		\draw[dotted] (\x,0) -- (\x,4.1);
		\foreach \y in {1,...,3}
		\draw[dotted] (10,\y) -- (16,\y);
		
		\draw[ultra thick, ->] (10,0)--(13.5,3.5);
		\draw[ultra thick, ->] (10,0)--(15.25,3.5);

		\filldraw (3,2) circle (2pt);
		\filldraw (13,2) circle (2pt);
		\filldraw (11,1) circle (2pt);

	\end{tikzpicture} 
	\]\caption{} \label{Fig:multiweighted}
	\end{figure}
\end{Ex}

\begin{Lem}
	$\cX_{\Sigma'}$ is a regular Deligne--Mumford stack, and $\pi$ is a proper, birational morphism.
\end{Lem}

\begin{proof}
	$\cX_{\Sigma'}$ admits an open cover by three open substacks, 
	referred to as \emph{charts}, corresponding to the three maximal cones 
	of $\Sigma'$. They are: \begin{align*}
		D_+(\langle \be_1,\bw \rangle) &= \left[\Spec(k[x_1',x_2',s,u][x_2'^{-1},u^{-1}]) \q \Gm^2\right] \simeq [\Spec(k[x_1',s]) \q \Gmu_{w_2}] \\
		D_+(\langle \bw,\bu \rangle) &= \left[\Spec(k[x_1',x_2',s,u][x_1'^{-1},x_2'^{-1}]) \q \Gm^2\right] \simeq [\Spec(k[s,u]) \q \Gmu_{w_2-kw_1}] \\
		D_+(\langle \bu,\be_2 \rangle) &= \left[\Spec(k[x_1',x_2',s,u][x_1'^{-1},s^{-1}]) \q \Gm^2\right] \simeq \Spec(k[x_2',u]).
	\end{align*}
	where: \begin{enumerate}
		\item the isomorphisms are supplied by \cite[Lemma 1.3.1]{Quek-Rydh},
		\item  $\zeta\in \Gmu_{w_2}$ acts on $\Spec(k[x_1',s])$ via $(x_1',s) \mapsto (\zeta^{w_1} x_1', \zeta^{-1}s)$.
		\item $\xi \in\Gmu_{w_2-\kappa w_1}$ acts on $\Spec(k[s,u])$ via $(s,u) \mapsto (\xi^{-1}s, \xi^{w_1}u)$.
	\end{enumerate}
	Since we have $ p \mid w_1 $ and $ p\nmid w_2 $, we have presented each chart as the stack-theoretic quotient of a smooth $k$-scheme under the action of a finite group whose order is invertible in $k$. So
	 	$\cX_{\Sigma'}$ is a smooth tame Deligne--Mumford stack over $k$, see \cite[Definition 2.3.1]{Abramovich-Vistoli}.
	 We noted $\pi$ is birational in Construction~\ref{Con:multi-weighted-blow-up}(3), and $\pi$ is proper because it can be factored into the composition of a coarse moduli space morphism $\cX_{\Sigma'} \to X_{\Sigma'}$ which is proper, followed by a toric morphism $X_{\Sigma'} \to \IA_\K^2$ which is also proper.
\end{proof}

Consider more generally a reduced center
 $\bar J = (x_1^{1/w_1},x_2^{1/w_2})$ on a regular surface $S$ defined over $k$, where $w_1,w_2$ are coprime integers, with the characteristic $p>0$ of $k$ dividing $w_1$. As before, the weighted blow-up of $S$ along $\bar J$ is an Artin stack, not a  Deligne--Mumford stack. To resolve this issue, we adapt the ideas developed in the  setting of Construction~\ref{Con:multi-weighted-blow-up}.

To motivate the forthcoming construction, recall from Definition \ref{Def:weighted} and the local Rees algebra presentation of the weighted blow-up in Equation \eqref{eq:presentation} that
\[
\Bl_{\bar J}(S) = \left[\Spec_S\left(\frac{\cO_S[x_1',x_2',s]}{(x_1's^{w_1} - x_1,x_2's^{w_2} - x_2)}\right) \smallsetminus V(x_1',x_2')\ \bigg/ \ \Gm\right] \to S,
\]
taking into account Remark \ref{Rk:local}.

\begin{Con}[Multi-weighted blow-up]
	We associate to $\bar J = (x_1^{1/w_1},x_2^{1/w_2})$ a \emph{multi-graded Rees algebra} $\tilde R'$
	on $S$ (in the sense of \cite[Section 2.4]{AQ}) and its associated multi-weighted blow-up. 
	
	We define $\tilde R'$ to be the $\IZ^2$-graded $\cO_S$-subalgebra
	of $\cO_S[s^{\pm 1},u^{\pm 1}]$, which at $q$ is generated by $s, u, x_1' := x_1s^{-w_1}u^{-1}$ and $x_2' := x_2s^{-w_2}u^{-\kappa}$, where $\kappa = \lceil w_2/w_1 \rceil$, and equals the full algebra $\cO_S[s^{\pm 1},u^{\pm 1}]$ elsewhere. The bi-grading of $s,u,x_1',x_2'$ is $(-1,0), (0,-1), (w_1,1)$, and $(w_2,\kappa)$, respectively.

		Finally consider 
		\[
		S'':= \left[\Spec_S(\tilde R') \smallsetminus V(x_2'u,x_1'x_2',x_1's) \ \big/ \ \Gm^2\right] \xrightarrow{\pi} S.
		\]

\end{Con}

	\begin{Lem}
		We have an isomorphism of $\IZ^2$-graded $\cO_{S,q}$-algebras: \[
		\tilde R'_q \simeq \frac{\cO_{S,q}[x_1',x_2',s,u]}{(x_1's^{w_1}u-x_1\,,\ x_2's^{w_2}u^\kappa - x_2)}
		\]
		where the $\IZ^2$-gradings on $x_1'$, $x_2'$, $s$, $u$ are given by $(w_1,1)$, $(w_2,\kappa)$, $(-1,0)$, $(0,-1)$ respectively. In particular it is the pullback of $\tilde R'_0$ along the morphism $S_q \to \IA_\K^2$ defined by $(x_1,x_2)$. 
	\end{Lem}

	\begin{proof}
		The indicated relations $x_1's^{w_1}u=x_1,\ x_2's^{w_2}u^k =x_2$ follow from the construction of $x_1', x_2'$. This provides a surjection ${\cO_{S,q}[x_1',x_2',s,u]}/{(x_1's^{w_1}u-x_1, x_2's^{w_2}u^\kappa - x_2)} \to \tilde R'_q$. 
		 Both algebras are $4$-dimensional and the algebra on the left is integral, hence the homomorphism is injective as well, as needed.
	\end{proof}

	\begin{Lem}
		$\tilde R'$ and $S''$ are  well-defined, independent of the choice of $x_1$ and $x_2$.
	\end{Lem}

\begin{figure}[h]
\begin{tikzpicture}[scale=0.6]
	
	
	\draw[->, thick] (0,0)--(8.5,0); 
	\draw[->, thick] (0,0)--(0, 6.75,0); 
	
	
	\node at (-0.7,5) {\small $ x_2^{w_1} $};
	
	\node at (6, -0.7) {\small $ x_1^{w_2} $};
	\node at (8.6, 0.6) {\small $ x_1^{w_2+\kappa} $};
	
	\node at (6.3,1.7) {\small $ x_1^{w_2} x_2 $};
	\node at (-1, 6) {\small $ x_2^{w_1+1} $};

	\foreach \x in {1,...,8} 
	\draw (\x,0.1) -- (\x,-0.1);
	
	\foreach \y in {1,...,6} 
	\draw (0.1,\y) -- (-0.1,\y);

	\foreach \x in {1,...,8}
	\draw[dotted, gray] (\x,0) -- (\x,6.4);
	\foreach \y in {1,...,6}
	\draw[dotted, gray] (0,\y) -- (8.3,\y);
	
	\draw[very thick, dashed] (0,5)--(6,0);
	\draw[very thick, -] (0,6)--(6,1)--(8,0);
	\draw[very thick, dotted] (6,1)--(8.4,-1);
	
	\node at (10.5,-1) {\small $ v_{\bar J} = (w_1+1)w_2 $};
	\node at (2.4,1.5) {\small $ v_{\bar J} = w_1 w_2 $};
	
	\filldraw (0,6) circle (3pt);
	\filldraw (6,1) circle (3pt);
	\filldraw (8,0) circle (3pt);
	
	\draw[->, thick, blue] (7,0.5)--(8,2.5); 
	\node[blue] at (8.2,3) {\small $ (1,\kappa)$};
	
	\draw[->, thick, blue] (3,3.5)--(4.75,5.6); 
	\node[blue] at (5,6) {\small $ \bw $};

\end{tikzpicture}
\caption{} 
\label{Fig:multi}
\end{figure}

	\begin{proof}
		Since the valuation $v_{\bar J}$ is canonically associated to $\bar J$, which we have shown to be well-defined, it suffices to show that $\tilde R'$ is canonically associated to $v_{\bar J}$. Also, for any choice of $x_1$ and  $x_2$ the multi-weighted $\tilde R'$ is the center associated to the stacky fan $(\Sigma',\beta')$, which itself is canonically associated to its underlying fan $\Sigma'$.
				
		We show that there is an ideal $I'$, canonically associated with $v_{\bar J}$,  such that for any choice of parameters $x_1, x_2$ such that $J = (x_1^{a_1},x_2^{a_2})$, the fan of the blow-up $\Bl_{I'}(S) \to S$ is $ \Sigma'$. In other words, $\tilde R'$ is independent of the choice of $x_1$ and  $x_2$ and depends only on $\bar J$, as needed.
		
		For any positive $r$ one defines $$I_r = \{g\in \cO_{S,q} \, : \, v_{\bar J}(g) \geq r\}.$$ This is an integrally closed ideal, monomial in any choice of $x_1, x_2$, canonically associated to $\bar J$. For instance we have $I_{w_1w_2} = (x_1^{w_2}, x_2^{w_1})^{\Int}$.  
		
		One identifies the ideal $I':= I_{w_1w_2+w_2} = (x_1^{w_2+\kappa}, x_1^{w_2}x_2, x_2^{w_1+1})^{\Int}$, see Figure \ref{Fig:multi}, whose dual fan is indeed $ \Sigma'$, as needed.

This implies also that $\Spec_S(\tilde R')$ is independent of choices. To show the same for $S''$, we prove that the irrelevant ideal $(x_2'u,x_1'x_2',x_1's)$ is independent of choices, by showing that it is naturally induced by the ideal $I'$ above. Indeed, observe that
$I' = (x_1^{w_2+\kappa},x_2^{w_1+1},x_1^{w_2}x_2)^{\Int}$
pulls back on $\Spec_S(\tilde R')$  to
$$s^{w_1w_2+w_2}u^{w_2+\kappa} \cdot (x_1'^{w_2+\kappa}s^{\kappa w_1-w_2},x_2'^{w_1+1}u^{\kappa w_1-w_2},x_1'^{w_2}x_2')^{\Int},$$
i.e. the controlled  transform 
$(x_1'^{w_2+\kappa}s^{\kappa w_1-w_2},x_2'^{w_1+1}u^{\kappa w_1-w_2},x_1'^{w_2}x_2')^{\Int} $
of $I'$
is independent of choices. 
Consequently the radical of this transform, which is $(x_1's,x_2'u,x_1'x_2')$, is independent of choices.
\end{proof}

\section{The invariant drops}\label{Sec:drops}

Having introduced the multi-weighted blow-up $ \pi \colon S'' \longrightarrow S $,
we show that the invariant drops, in fact the order of vanishing drops,  at every point above  $q$.
This provides Theorem~\ref{Thm:Main2}\eqref{It:invariant-drops-p}
and hence concludes the proof of our main theorems.

Recall that  we consider
$ \bw = (w_1, w_2) $ and, without loss of generality, $ p \mid w_1 $ 
(and $ p \nmid w_2 $ since $ w_1, w_2 $ are coprime).
As in Section \ref{Sec:multi}, and in contrast to earlier sections, we allow both $ w_1 > w_2 $ and $ w_1 < w_2 $.

Further recall that the multi-weighted blow-up $ \pi $ is given by the transformation
\[
	x_1 \mapsto x_1's^{w_1}u \qquad \textrm{and} \qquad x_2 \mapsto x_2's^{w_2}u^\kappa 
\]
where 
$ \kappa = \lceil w_2/w_1 \rceil $.

	In analogy to Definition \ref{Def:proper}, we introduce:
	\begin{Def} The \emph{proper transform}  $f'$ of $f$  
	is determined by
			$$f( x_1's^{w_1} u, \,  x_2's^{w_2} u^\kappa ) \ \ = \ \ s^\alpha u^\beta f'(x_1',x_2',s,u), \qquad s,u\ \nmid\ f'(x_1',x_2',s,u).$$
	\end{Def}

For detecting the improvement of the singularity in the different charts of $ \pi $,
the initial forms of the local generator $ f $ of $ C $ at $ q $ with respect to the vectors $ \bw $ and $ \bu =  (1,\kappa ) $ play an essential role.
Hence, let us introduce this notion.

\begin{Def}
	Let $ (\cO, \mm, \K) $ be a $ 2 $-dimensional regular local ring containing its residue field $ \K $.
	Let $ (x, y ) $ be regular system of parameters for $ \cO $. 
	Let $ f \in \cO \smallsetminus \{ 0  \}$ be a non-zero element with $ \mm $-adic expansion
	$ f = \sum_{(a,b) \in \IZ_{\geq 0}^2} c_{a,b} x^a y^b $ with $ c_{a,b} \in \K $.
	For an vector $ \bv = (v_1, v_2) \in \IZ_{\geq 0}^2 $,
	the  {\em $ \bv $-order $ \nu_\bv (f) $} and
	the {\em $\bv $-initial form  $ \ini_\bv (f) $ of $ f $} are defined as
	\[
		\nu_\bv (f) := \inf\{ a v_1 + b v_2  \mid (a,b) \in \IZ_{\geq 0}^2 : c_{a,b} \neq 0 \}
	\]
	and
	\[
		\ini_\bv (f) := \sum_{(a,b) \, :  \, a v_1 + b v_2 = \nu_\bv (f) } c_{a,b} X^a Y^b \in \K[X,Y].
	\] 
	In analogy with Definition \ref{Def:Fnu_Fdelta}, we define the lift of $ \ini_\bv(f) $ to $ \cO $ by
	\[
		f_\bv :=  \sum_{(a,b) \, :  \, a v_1 + b v_2 = \nu_\bv (f) } c_{a,b} x^a y^b \in \cO.
	\]
\end{Def}

The  lift $ f_\bv $ is the sum over those terms in the expansion of $ f $
which contribute to the face of the Newton polyhedron of $ f $ whose supporting hyperplane has normal vector $ \bv $.

We illustrate how this notion is used to detect the strict decrease of the order in the following examples.

\begin{Ex}
	\label{Ex:Dan}
	Let $ \K =\IF_3 $, $ q \in \IA^2$ the origin and $ C = \{ f = 0 \} \subset \IA^2 $ for  
	\[
		f = x_2^7 + x_1^4 x_2^4 + x_1^7 x_2^2 + x_1^9 x_2 + x_1^{11} + x_1^{6} x_2^3.  
	\]
	In this example, $a_2 = 7 < a_1 = 28/3$, so $ \bw = (w_1, w_2) = (3,4) $ (thus $ p = 3 \mid w_1 $, as assumed)
	and $ \bu = (1, \kappa) = (1,2) $.
	The Newton polyhedron of $ f $ is shown in Figure~\ref{Fig:example_initials}.
	In particular, the supporting hyperplanes defined by $ \bw $ resp.~$ \bv $ are marked.
	Observe that the intersection point of these two hyperplanes is {\em not} a lattice point, i.e., it is not in $ \IZ_{\geq 0}^2 $.
	
	\begin{figure}[h]
		\begin{tikzpicture}[scale=0.8]
			
			\filldraw[lightgray] (11.5,0)--(11,0)--(9,1)--(7,2)--(4,4)--(0,7)--(0,7.5)--(11.5,7.5);
			
			\draw[->, thick] (0,0)--(11.75,0); 
			\draw[->, thick] (0,0)--(0, 7.75); 
			
			\node at (12.3,-0.1) {$ x_1 $};
			\node at (-0.5,7.8) {$ x_2 $};
			
			\node at (-1,7) {$a_2 = 7 $};
			\node at (-0.4,4) {$ 4 $};
			\node at (-0.4,3) {$ 3 $};
			\node at (-0.4,2) {$ 2 $};
			\node at (-0.4,1) {$ 1 $};
			
			\node at (4, -0.4) {$ 4 $};
			\node at (6, -0.4) {$ 6 $};
			\node at (7, -0.4) {$ 7 $};
                \node at (9.33, -0.6) {$a_1 = 28/3$};
			\node at (11, -0.4) {$ 11 $};
			
			\foreach \x in {1,...,11} 
			\draw (\x,0.1) -- (\x,-0.1);

                \draw (9.33,0.2)--(9.33,-0.2);
			
			\foreach \y in {1,...,7} 
			\draw (0.1,\y) -- (-0.1,\y);

			\foreach \x in {1,...,11}
			\draw[dotted, gray] (\x,0) -- (\x,7.6);
			\foreach \y in {1,...,7}
			\draw[dotted, gray] (0,\y) -- (11.6,\y);
			
			\draw[thick, dotted] (11.5,0)--(11,0)--(9,1)--(7,2)--(4,4)--(0,7)--(0,7.5);
			\draw[very thick, dashed] (0,7)--(10,-0.5);
			\draw[very thick] (11,0)--(5,3);
			
			\filldraw (11,0) circle (3pt);
			\filldraw (9,1) circle (3pt);
			\filldraw (7,2) circle (3pt);
			\filldraw (4,4) circle (3pt);
			\filldraw (0,7) circle (3pt);
			\filldraw (6,3) circle (3pt);
			
			\draw[->, thick, blue] (7.5,1.75)--(8.5,3.75); 
			\node[blue] at (8.7,3.95) {$ \bu $};
			
			\draw[->, thick, blue] (2,5.5)--(3.5,7.5); 
			\node[blue] at (3.9,7.2) {$ \bw $};

		\end{tikzpicture}
		\caption{} 
		\label{Fig:example_initials}
	\end{figure}

	We observe that 
	\[
	\begin{array}{l}
		f_\bw = x_2^7 + x_1^4 x_2^4 = x_2^4 ( x_2^3 + x_1^4 ),
		\\[5pt]
		f_\bu =   x_1^7 x_2^2 + x_1^9 x_2 + x_1^{11} = x_1^7 (x_2^2 + x_1^2 x_2 + x_1^4). 
	\end{array}
	\]
	Clearly, $ f_\bw = f_{J,1} $.

	The proper transform  is
	\[
		f' =
		x_2'^7 u^3 + x_1'^4 x_2'^4 u + x_1'^7 x_2'^2 s + x_1'^9 x_2' s^3 + x_1'^{11} s^5 + x_1'^{6} x_2'^3 s^2 u  	.
	\]
	The polynomials corresponding to the $ f_\bw $ resp.~$ f_\bu $ are seen by setting $ s= 0$ in $ f' $ resp. $ u = 0 $:
	\[
	\begin{array}{lcl}
		f_\bw^\star  := x_2'^7 u^3 + x_1'^4 x_2'^4 u 
		 & \mbox{ resp. }&
		f_\bu^\star :=  x_1'^7 x_2'^2 s + x_1'^9 x_2' s^3 + x_1'^{11} s^5. 
	\end{array}
	\]
	
	The multi-weighted blow-up has three charts,
	where
	$ x_1' s \neq 0$,
	$ x_2' u \neq 0$,
	$ x_1' x_2' \neq 0$, respectively.

	First suppose 
	$ x_1' s \neq 0$.
	We take the \'etale slice $ s = x_1' = 1 $.
	As the situation is unchanged away from the center, 
	we can focus on the points on the exceptional divisors of which only $ \{ u=0\} $ is seen in this chart. 
	The log-order with respect to $ u $ provides an upper bound for the order at points on $ \{ u = 0 \} $.
	Hence, we may plug in $ u = 0 $ and get to 
	$ 
		x_2'^2 + x_2' + 1.
	$ 
	This is the same as plugging in $ s = x_1' = 1 $ in $ f_\bu^\star $. 
	The order of this polynomial is at most $ 2 $ which is strictly smaller than $ \nu_q (C) = 7 $, as desired.
	
	Next consider the chart where
	$ x_2' u $ is invertible.
	As in the previous chart, 
	we pass to the \'etale slice $ x_2' = u = 1 $
	and consider the log-order with respect to the exceptional variable $ s $. 
	We get the polynomial
	$ 
		1 + x_1'^4
	$ 
	which is connected to $ f_\bw^\star $. 
	Again, we see that the order is bounded by a value strictly smaller than $ \nu_q (C) $.
	
	Finally, we turn to the last chart, $ x_1' x_2' $ invertible.
	Here, we plug in $ x_1' = x_2' =  1 $ and get
	$ 
		u^3 + u + s +  s^3 + s^5 + s^2 u  	.
	 $ 
	 	It only remains to consider the point $ \{ s = u = 0 \} $,
	at which the order is $ 1 $ and hence also strictly smaller than $ 7 $, as needed.
	 
	 In conclusion, the order dropped in every chart below $ 7 $.
 \end{Ex}

We briefly look into a variant of Example~\ref{Ex:Dan}

\begin{Ex}
	\label{Ex:Dan_modified}
	Let $ \K =\IF_3 $, $ q \in \IA^2$ the origin and $ C = \{ g = 0 \} \subset \IA^2 $ for  
	\[
	g = x_2^7 + x_1^4 x_2^4 + x_1^8 x_2 + x_1^7 x_2^2 + x_1^9 x_2 + x_1^{11} + x_1^{6} x_2^3.  
	\]
	We still have $ \bw = (w_1, w_2) = (3,4) $ and $ \bu = (1,2) $.
	But now, we get
	\[
	\begin{array}{lcl}
		g_\bw = x_2^7 + x_1^4 x_2^4  + x_1^8 x_2 
		&
		\mbox{ and }
		&
		g_\bu =  x_1^8 x_2. 
	\end{array}
	\]
	
	In the chart where $ x_2' u $ is invertible,
	we substitute $ x_2' = u = 1 $ and the log-order with respect to $ s $	
	leads us to the polynomial
	\[
		\gamma(x_1'): = 1 + x_1'^4 + x_1'^8
	\]
	coming from $ g_\bw $.
	Due to the term $ x_1'^8 $ one might think that the argument that the order is strictly smaller than $ 7 $ does not apply anymore. 
	This is not the case. We note that $\gamma(x_1') = \tilde\gamma(x_1'^4)$, where $\tilde\gamma(x) = 1+ x+x^2$ has degree 2. Factoring $\tilde\gamma(x) = \prod (x-\alpha_i)^{m_i}$, we have that $\alpha_i \neq 0$ and $m_i\leq 2$. Since $p=3\nmid  4$, we have that $x_1'^4 - \alpha_i$ is separable, hence the multiplicity of every root of $\gamma$ is $\leq 2 < 7$, as needed. In the general argument, this step is subsumed in the proof of Theorem \ref{Thm:Main1}\eqref{It:order-drops}.
\end{Ex}

In general, we have:

\begin{proof}[Proof of Theorem~\ref{Thm:Main2}\eqref{It:invariant-drops-p}: the invariant drops in the multi-weighted blow-up] \hfill

We prove that the order drops, which, in the Deligne--Mumford setting, implies that the invariant drops.

	Set $ \nu := \nu_q (C) $.	
	The multi-weighted blow-up has three charts,
	where
	$ x_1' s \neq 0$,
	$ x_2' u \neq 0$,
	$ x_1' x_2' \neq 0$, respectively.

	First suppose 
	$ x_1' s \neq 0$.
	We take the \'etale slice $ s = x_1' = 1 $.
	As the situation is unchanged away from the center, 
	we can focus on the points on the exceptional divisors of which only $ \{ u=0\} $ is seen in this chart. 
	The log-order with respect to $ u $ provides an upper bound for the order at points on $ \{ u = 0 \} $.
	Hence, we may plug in $ u = 0 $ and get to 
	$ 
		 f' (1,x_2',1,0).
	$ 
         The order of this polynomial is $<a_2$. If $\nu=a_2<a_1$ we are done. However, if $a_1<a_2$, then in fact the monomial $x_1^{a_1}$ appears in $f$ with non-zero coefficient, and $x_1'^{a_1}$ is the only monomial appearing with non-zero coefficient in $f'(x_1',x_2',s,0)$. Thus, $f'(1,x_2',1,0) $ is that non-zero constant, and there is nothing to prove.

	Next consider the chart where
	$ x_2' u \neq 0$.
	As in the previous chart, 
	we pass to the \'etale slice $ x_2' = u = 1 $
	and consider the log-order with respect to the exceptional variable $ s $.

	We get the polynomial
	$ 
		f'(x_1',1,0,1).
	$ 
	This is the polynomial appearing in the corresponding chart of the weighted blow-up. We have shown that the order drops on the weighted blow-up in Theorem \ref{Thm:Main1}\eqref{It:order-drops}.
	
	Finally, we turn to the last chart $ x_1' x_2'\neq 0 $.
	Here, we plug in $ x_1' = x_2' =  1 $ and get
	$ 
		f'(1,1,s,u).
	 $ 
	 To complete the proof we need only check the order at the point where $ \{ s=u=0 \} $.
		
		The terms that we can control are those coming from the initial forms $ f_\bw $ and $ f_\bu $. 
		Notice that these polynomials lie in $ \K[u] $ resp.~$ \K[s] $. 
		If $ \nu_1 = a_1 < a_2 $ then $ x_1^{a_1} $ appears with non-zero coefficient in $ f(x_1,x_2) $ as well as in $ f_\bw(x_1,x_2) $.
		Since $ x_1 $ is substituted by $ x_1' s^{w_1} u $,
		the monomial $ u^{a_1 - \nu_\bu(f)} $ appears with non-zero coefficient in $  f'(1,1,s,u) $ and especially without an $ s $-factor.
		This provides a bound of the order at $ \{s=u=0\} $ which is strictly smaller than $ \nu_q (C) = \nu = a_1 $ since  $ \nu_\bu(f)>0 $.
		
		Suppose $ a_1 > a_2 = \nu $.
		Therefore, $ x_2^{a_2} $ has non-zero coefficient in the expansion of $ f(x_1, x_2) $ and in $ f_\bw(x_1, x_2 ) $.
		By definition, we have 
		for every $ x_1^\alpha x_2^\beta $ which appears with non-zero coefficient in an expansion of $ f $:
		\[
			 w_1 \alpha + w_2 \beta \geq \nu_\bw (f) = a_1 w_1 = a_2 w_2.
		\]
		  Applying the above inequality to an $(\alpha,\beta)$ for which $\nu_\bu(f) = \alpha+\kappa\beta$, and using that $w_1$ and $w_2$ are coprime, we obtain:
		\[
		\nu_\bu(f)
		= 
		\alpha + \left\lceil \frac{w_2}{w_1}\right\rceil  \beta
		 > 
		 \alpha + \frac{w_2}{w_1}  \beta
		 =   \frac{w_1 \alpha + w_2 \beta }{w_1} \geq  a_1.
		\]
        Since we plug in
		$x_2's^{w_2}u^\kappa $
		for $ x_2 $,
		we obtain from $ x_2^{a_2} $ in $  f'(1,1,s,u) $ the term
		$ u^{\kappa a_2 - \nu_\bu(f) } $ without an $ s $-factor.
		We observe that 
		\[
			\kappa a_2 - \nu_\bu(f)
			< \kappa a_2 - a_1 
			= \left( \left\lceil \frac{w_2}{w_1}\right\rceil - \frac{w_2}{w_1} \right) a_2 
			< a_2,
		\]
		where $\lceil {w_2}/{w_1} \rceil - {w_2}/{w_1} < 1 $ by definition.
		In Figure~\ref{Fig:nu_u} we illustrate how the difference $ 
			\kappa a_2 - \nu_\bu(f) $ arises in the Newton polyhedron of $ f $. 
			(Note that we only draw the faces of the Newton polyhedron determined by $ \bu $ and $ \bw $.)
			\begin{figure}[h]
			\begin{tikzpicture}[scale=0.65]
				
				\filldraw[black!15!white] (12.5,0)--(12,0)--(9,1)--(7.5,1.5)--(3,4)--(0,7)--(0,7.5)--(12.5,7.5);
				
				\draw[->, thick] (0,0)--(12.75,0); 
				\draw[->, thick] (0,0)--(0, 7.75); 
				
				\node at (13.3,-0.1) {$ x_1 $};
				\node at (-0.6,7.8) {$ x_2 $};
				
				\draw[very thick] (7,0.1) -- (7,-0.2);
				\node at (7, -0.5) {$ a_1 $};
				\draw[very thick] (0.1,7) -- (-0.2,7);
				\node at (-0.6,6.8) {$ a_2 $};
				
				\draw[very thick]
				(0,7.5) -- (0,7) -- (3,4);
				\draw[very thick, dotted]
				(3,4)--(7,0);
				
				\draw[very thick] (12.5,0)--(12,0)--(7.5,1.5);
				\draw[very thick, dashed] (7.5,1.5)--(-1.5,4.55);
				\draw[very thick, dashed] (12,0)--(13.5,-0.5);
				
				\draw[very thick, dashed] (-1.5,7.55) --(13.5,2.5);

				\node at (-3.5,4.6) {$ \nu_\bu (.) = \nu_\bu (f)$};
				\node at (-3.25,7.6) {$ \nu_\bu (.) = \kappa a_2 $};

				\filldraw (12,0) circle (3pt);
				\filldraw (7.5,1.5) circle (3pt);
				\filldraw (3,4) circle (3pt);
				\filldraw (0,7) circle (3pt);
				
				\draw[<->,very thick, blue] (5.25,2.25)-- (6.16,4.98);
				\node[blue] at (7.4,3.2) {$ \kappa a_2 - \nu_\bu (f) $};

				\draw[->, thick, gray] (10.5,0.5)--(11,2); 
				\node[gray] at (11.5,2) {$ \bu $};
				
				\draw[->, thick, gray] (1.5,5.5)--(2.5,6.5); 
				\node[gray] at (3,6.5) {$ \bw $};

			\end{tikzpicture}
			\caption{} 
			\label{Fig:nu_u}
		\end{figure}
		 
		So  the order of $ f' $ at the point $ \{ s = u = 0 \} $ is strictly smaller than $ \nu = \nu_q (C)  $.
		The assertion follows.
\end{proof}

\end{document}